\documentclass{article}

\usepackage{amssymb, amsmath}
\usepackage{mathrsfs} 
\usepackage{color}
\usepackage{pgf,tikz}
\usepackage{graphicx}
\usepackage{hyperref}
\newenvironment{proof}{\noindent \textbf{Proof.}}{\hfill $\square$}
\numberwithin{equation}{section}
\makeatletter
\newcommand{\labeltarget}[1]{\Hy@raisedlink{\hypertarget{#1}{}}}
\makeatother

\newtheorem{theorem}{Theorem}[section]
\newtheorem{lemma}[theorem]{Lemma}
\newtheorem{proposition}[theorem]{Proposition}
\newtheorem{corollary}[theorem]{Corollary}
\newtheorem{remark}[theorem]{Remark}
\newtheorem{definition}[theorem]{Definition}

\title{Signed and sign-changing solutions for a class Kirchhoff-type problem involving the fractional p-Laplacian with critical Hardy nonlinearity  }

\author{
	R. F. Gabert\thanks{R. F. Gabert was supported 
		in part by  CAPES. Universidade Federal de S\~ao Carlos, Departamento de Matem\'atica, 
		CEP. 13565-905, S\~ao Carlos, SP - Brazil. E-mail:rodrigogabert@dm.ufscar.br}; \,
	{R.S. Rodrigues\thanks{Universidade Federal de S\~ao Carlos, Departamento de Matem\'atica, 
			CEP. 13565-905, S\~ao Carlos, SP - Brazil. E-mail: rodrigo@dm.ufscar.br}
		\thanks{Corresponding author.}}}

\date{}

\begin{document}
	\pretolerance10000 

	\maketitle
	
\begin{abstract}
In this paper, we study the existence of three solutions for a Kirchhoff equation involving the nonlocal fractional p-Laplacian considering Sobolev and Hardy nonlinearities at subcritical and critical growths. The proof is based on Mountain Pass Theorem, Ekeland's variational principle and constrained minimization in Nehari sets.
\end{abstract}

\section{Introduction}\label{Section1}

The proposal of this paper is to establish the existence of a positive solution, a negative one and a sign-changing one of following fractional Kirchhoff problem:

\begin{equation}
\tag{$ P_{\lambda} $}
\left \{
\begin{array}{ll}
M\left (\int _{\mathbb{R}^{2N}}\frac{|u(x)-u(y)|^{p}}{|x-y|^{N+ps}}dx dy\right )(-\Delta_{p})^s u=\lambda f(x,u)+\frac{|u|^{q-2}u}{|x|^{\alpha}} & \operatorname{in}\, \Omega, \\
u=0  &\operatorname{in}\,  \mathbb{R}^{N} \setminus \Omega,\\
\end{array}
\right.
\label{Equa1}
\end{equation}
where $ \Omega \subset \mathbb{R}^{N} $ is a smooth bounded domain containing $ 0 $, $ 0<s<1 $, $ 0\leq \alpha< ps<N $, $ \lambda >0 $ and $ 1< p<q\leq p_{\alpha}^{*} $ being $ p_{\alpha}^{*}:=\frac{(N-\alpha)p}{N-sp} $ the fractional critical Hardy-Sobolev exponent. Also, $ M:[0,\infty)\rightarrow [0,\infty) $ is a
continuous function and $ f:\Omega \times \mathbb{R}\rightarrow \mathbb{R}$ is a Carath\'eodory function whose assumptions over them will be introduced later and $ (-\Delta_{p})^{s} $ represents the factional p-Laplacian operator, which for $ u\in C_{c}^{\infty}(\mathbb{R}^{N}) $  is defined as
\begin{equation*}
(-\Delta_{p})^{s}u(x)=C(N,s)\lim_{\epsilon \rightarrow 0}\int_{\mathbb{R}^{N}\setminus B(x,\epsilon)}\frac{|u(x)-u(y)|^{p-2}(u(x)-u(y))}{|x-y|^{N+ps}}dx dy,\, x\in \mathbb{R}^{N},
\end{equation*} 
for $ x\in \mathbb{R}^{N} $, where $ B(x,\epsilon):=\{y\in \mathbb{R}^{N}:|x-y|<\epsilon\} $ and $ C(N,s) $ is a positive normalizing constant.

When $ p=2 $, the problem \eqref{Equa1} is a fractional version of a classical stationary Kirchhoff problem
\begin{equation}\label{Equa2}
-M\left (\int _{\Omega}|\nabla u|^{2}dx\right )\Delta u= g(x,u),\, \operatorname{in}\, \Omega,
\end{equation} 
where $ \Omega\subset \mathbb{R}^{N} $ is a smooth domain, with $ u $ satisfying some boundary conditions and $ g $ some growth conditions. The physical motivation for the study of this type of problem is the Kirchhoff equation modeling nonlinear vibrations 
\begin{equation*}
\left \{
\begin{array}{ll}
u_{tt}-M\left (\int _{\Omega}|\nabla u|^{2}dx\right )\Delta u= g(x,u) & \operatorname{in}\, \Omega\times (0,T), \\
u=0  &\operatorname{on}\,  \partial\Omega\times (0,T),\\
u(x,0)=u_{0}(x),u_{t}(x,0)=u_{1}(x),
\end{array}
\right.
\end{equation*}  
as a generalization of the stationary analogue of the equation proposed by Kirchhoff in \cite{Kir}
\begin{equation}\label{Equa61}
\rho \frac{\partial ^{2}u}{\partial t^{2}}-\left (\frac{P_{0}}{h}+\frac{E}{2L}\int _{0}^{L}\left |\frac{\partial u}{\partial x}\right |^{2}dx\right )\frac{\partial ^{2}u}{\partial x^{2}}=0
\end{equation}
which takes into account the effects of the changes in the length of the string produced from transverse vibrations. In equation \eqref{Equa61}, L is the length of the string, h is the area of cross-section, E is the Young modulus of the material, $ \rho $ is the mass density and $ P_{0} $ is the initial tension. The Kirchhoff's equation is a nonlinear extension of D'Alembert’s wave equation.

In the last years, a great attention have been devoted to nonlocal operators of the fractional type. These operators appear in several areas of knowledge such as mathematical finances, quantum mechanics, water waves, phase transition, minimal surface, population dynamics, optimal control, game theory, L\'{e}vy processes in probability theory, among others. For more details about these subjects and them applications see \cite{F, B, Caf, RO} and the references therein. As the studies about fractional problems were spreading, it also made sense to consider the presence of a Kirchhoff term in these equations. In \cite{Fi}, a physical motivation was proposed taking into account the nonlocal aspect of the tension in the string in the model proposed by Kirchhoff.  

Over the years, in the context of classical Laplacian operator, many authors dealt with problems of Kirchhoff proving results concerning existence, multiplicity and information about the signal of solutions through variational methods. Among those which prove existence of non-trivial and positive solutions we can highlight \cite{E}, including in the critical case as \cite{FG,FJ}. However, to find sign-changing solutions it seem to be more complicated as we can see in \cite{FF}. In that paper, the existence of a least energy sign-changing solution for problem \eqref{Equa2} was obtained in bounded domain $ \Omega\subset \mathbb{R}^{3} $, assuming $ M $ of class $ C^{1} $, increasing, non-degenerate ($ M(0)>0 $) with $ \frac{M(t)}{t} $ decreasing in $ t> 0 $ and $ g $ of class $ C^{1} $, subcritical growth satisfying 4-superlinear Ambrosetti-Rabinowitz type condition and
\begin{equation*}
	  \frac{g(t)}{t^{3}}  \operatorname{is}\; \operatorname{increasing}\; \operatorname{in}\;  t>0 \; \operatorname{and} \;\operatorname{decreasing} \operatorname{in}\;  t<0 .
\end{equation*}
In \cite{NN}, assuming similar conditions for $ M $, but more flexible as $ \frac{M(t)}{t^{\frac{\mu-2}{2}}} $ decreasing in $ t>0 $ for some $ \mu>2 $ and $ g=\lambda f $, $ \lambda>0 $ with $ f $ of subcritical growth, without 4-superlinear Ambrosetti-Rabinowitz condition and satisfying
\begin{equation*}
\frac{g(t)}{|t|^{\mu-2}t}  \operatorname{is}\; \operatorname{increasing}\; \operatorname{in}\;  t>0 \; \operatorname{and} \;\operatorname{decreasing} \operatorname{in}\;  t<0,
\end{equation*}  
a ground state solution and a least energy sign-changing solution were obtained in bounded domain $ \Omega\subset \mathbb{R}^{N} $, $ N\geq 3 $, expanding the main result of \cite{FF}. Furthermore, assuming $ M $ only increasing, a technique of truncation of function $ M $ proposed in \cite{FG} was used to obtain the same results for $ \lambda $ large enough. In the context of fractional Kirchhoff problems, we cite the articles \cite{Luo, Cheng} for sign-changing solutions.

Even with the advances in the studies of existence of sign-changing solutions, there are few works which address this problem involving critical growth nonlinearities. We highlight papers \cite{Ta,MF} for solutions of this type involving Laplacian operator. Recently, existence of least energy sign-changing solutions for fractional Brezis-Nirenberg problem, number of sign changes and asymptotic behavior of these solutions were established in \cite{COR}. In \cite{K2}, the authors obtained signed ground state and least energy sign-changing solutions for a problem involving fractional p-Laplacian and critical Hardy-Sobolev nonlinearities. The techniques developed and adapted in that paper will be strongly used in our paper. 

Motivated by the articles cited above, more specifically, by \cite{K2,FF,NN}, we are interested in obtaining three solutions with sign information for problem \eqref{Equa1} with emphasis on the critical case $ q=p_{\alpha}^{*} $ and $ f $ satisfying some conditions similar to those required in \cite{NN} and compatible with the fractional p-Laplacian operator.

In view of our problem, we assume that function $ M: [0,\infty)\rightarrow [0,\infty) $ is continuous with $ \overline{M}(t):=M(t)t $ continuously differentiable in $ [0,\infty) $ and it satisfies the following conditions:
\begin{itemize}
\item[$ (M_{1}) $]\labeltarget{M1} $ M $ is increasing;
\item[$ (M_{2}) $]\labeltarget{M2} there exists $ \gamma \in (p,q) $ such that $ \frac{M(t)}{t^{\frac{\gamma-p}{p}}} $ is decreasing in $ t>0 $.
\end{itemize}

A prototype for $ M $ is given by $ M(t)=a+bt^{\vartheta} $ for
$ t \geq 0 $, with $ a, b \geq 0 $, $ a+b>0 $ and $ 0<\vartheta<\frac{q-p}{p} $. When $ M(t) \geq m_{0}>0 $ for all $ t \geq 0 $, Kirchhoff problems are said to be non-degenerate and in our model this occurs whenever $ a > 0 $ and $ b \geq 0 $. On the other hand, if $ M(0) = 0 $ and $ M(t) > 0 $ for all $ t > 0 $, the Kirchhoff problems are called degenerate which is the case of our particular $ M $ when $ a = 0 $ and $ b > 0 $.

In this paper, we deal with the non-degenerate and degenerate problems. Among the papers in which both the cases were dealt for fractional operators we highlight, for example, \cite{A1} in bounded domains with focus in the existence of positive solutions, and \cite{CP,FP} in $ \mathbb{R^{N}} $ whole with focus in the existence of non-trivial solutions.

Also, we assume that nonlinearity $ f:\Omega \times \mathbb{R}\rightarrow \mathbb{R}$ is a Carath{\'e}odory function such that $ \overline{f}(x,t):=f(x,t)t $ is continuously differentiable in the variable $ t $, for a.e. $ x\in \Omega $. Furthermore, in the first part of the paper, where we deal the degenerate case, we assume that $ f $ satisfies the following conditions:

\begin{itemize}
\item[($ f_{0} $)] \labeltarget{0} $ \overline{f}(x,t)>0 $, for a.e. $ x\in \Omega $, for all $ t\neq 0 $ and there exists $ C>0 $ such that 
\begin{equation*}
|\partial_{t}\overline{f}(x,t)|\leq C(1+|t|^{p^{*}-1}),
\end{equation*}
for a.e. $ x\in \Omega $ and all $ t\in \mathbb{R} $, where $ p^{*}:=p_{0}^{*}=\frac{Np}{N-sp} $ is the fractional critical Sobolev exponent.

\item[($ f_{1} $)] \labeltarget{1}  given $ \epsilon >0 $ and $ r\in (p,p^{*}) $, there exists $ C_{\epsilon,r}>0 $ such that 
\begin{equation*}
|f(x,t)|\leq \epsilon (|t|^{\gamma-1}+|t|^{p^{*}-1})+C_{\epsilon,r}|t|^{r-1},
\end{equation*} 
for a.e. $ x\in \Omega $ and all $ t\in \mathbb{R} $.

\item[($ f_{2} $)] \labeltarget{2} there exists $ \mu\in(\gamma,q) $ such that 
\begin{equation*}
	\mu F(x,t)\leq f(x,t)t
\end{equation*}
for a.e. $ x\in \Omega $ and all $ t\in \mathbb{R} $, where $ F(x,t):=\int _{0}^{t}f(x,\tau)d\tau $.

\item[($ f_{3} $)] \labeltarget{3} $ \frac{f(x,t)}{|t|^{\gamma-2}t} $ is increasing in $ t>0 $ and decreasing in $ t<0 $, for a.e. $ x \in \Omega $.
\end{itemize}

In the second part of the paper, we work in the non-degenerate case. One of the advantages of $ M $ to have a positive boundedness from bellow is the possibility of to remove the hypothesis \hyperlink{M2}{$(M_{2}) $}, keeping only \hyperlink{M1}{$(M_{1}) $} and even so to obtain a similar result for $ \lambda>0 $ large enough, making a truncation in the function $ M $ as made in \cite{NN}. For the non-degenerate case, we assume that $ f $ satisfies \hyperlink{0}{$ (f_{0}) $} and the following conditions:
\begin{itemize}
	\item[($ F_{1} $)]\labeltarget{F1} given $ \epsilon>0 $ and $ r\in (p,p^{*}) $, there exists $ C>0 $ such that
	\begin{equation*}
	|f(x,t)|\leq \epsilon (|t|^{p-1}+|t|^{p^{*}-1})+C_{\epsilon,r}|t|^{r-1},
	\end{equation*}
	for a.e. $ x\in \Omega $, for all $ t\in \mathbb{R} $.
	\item[($ F_{2} $)]\labeltarget{F2} there exists $ \mu \in (p,q) $ such that $\frac{f(x,t)}{|t|^{\mu-2}t} $ is increasing in $ t>0 $ and decreasing in $ t<0 $, for a.e. $ x\in \Omega $.
\end{itemize}
In this case, \hyperlink{F1}{$ (F_{1}) $} replaces \hyperlink{1}{$ (f_{1}) $} and \hyperlink{F2}{$ (F_{2}) $} replaces both \hyperlink{2}{$ (f_{2}) $} and \hyperlink{3}{$ (f_{3}) $}.
\vspace{0,6cm}

\textbf{Examples:}
\begin{itemize}
	\item[(a)] $ M(t)=\arctan\left (t^{\vartheta}\right ) $, $ 0<\vartheta<\frac{q-p}{p} $, is a degenerate Kirchhoff function which satisfies \hyperlink{M1}{$ (M_{1}) $}-\hyperlink{M2}{($ M_{2} $)} and
	\begin{equation*}
	f(x,t)=a(x)|t|^{r-2}t ,\;  p(\vartheta+1)<r<p^{*} ,\; a\in L^{\infty}(\Omega) ,\; a(x)\gneqq 0\;  \operatorname{a.e.}  x\in \Omega,
	\end{equation*}
	satisfies the conditions \hyperlink{0}{$ (f_{0}) $}-\hyperlink{3}{$ (f_{3}) $}.
	\item[(b)] Examples of non-degenerate Kirchhoff function $ M $ satisfying \hyperlink{M1}{$ (M_{1}) $}-\hyperlink{M2}{$ (M_{2}) $} and a function $ f $ which is not odd and satisfies \hyperlink{0}{$ (f_{0}) $}, \hyperlink{F1}{$ (F_{1}) $} and \hyperlink{F2}{$ (F_{2}) $} are, respectively, $ M(t)=m_{0}+\log(1+t^{\vartheta})$, $ m_{0}>0 $, $ 0<\vartheta<\frac{q-p}{p} $ and  
	\begin{equation*}
	f(t)=r\log(1+t^{+})(t^{+})^{r-1}+\frac{|t^{-}|^{r-1}t^{-}}{1+|t^{-}|} , \; p<r<p^{*},\;\; t^{+}:=\max\{t,0\},\;  t^{-}:=\min\{t,0\} .
	\end{equation*}
\end{itemize}

Our main result related to degenerate case is as follows:
\begin{theorem}\label{Theo11}
	Suppose that $ M(0)=0 $, $ M $ satisfies \hyperlink{M1}{$( M_{1}) $}-\hyperlink{M2}{$ (M_{2}) $} and $ f $ satisfies \hyperlink{0}{($ f_{0} $)}-\hyperlink{3}{($ f_{3} $)}. Then the following statements hold true:
	\begin{itemize}
		\item[(1)] if $ p<q<p_{\alpha}^{*} $, we have that, for all $ \lambda >0 $
		\begin{itemize}
			\item[($ \mathcal{C}_{\lambda} $)]\labeltarget{C}	 
			problem \eqref{Equa1} has a positive solution $ u_{\lambda,1} $ and a negative one $ u_{\lambda,2} $ such that one of them is a ground state solution. Furthermore, problem \eqref{Equa1} has a least energy sign-changing solution $ u_{\lambda,3} $ such that 
			\begin{equation*}
			I_{\lambda}(u_{\lambda,3})>I_{\lambda}(u_{\lambda,1})+I_{\lambda}(u_{\lambda,2}),
			\end{equation*}
			where $ I_{\lambda} $ is the energy functional associated to this problem.
		\end{itemize}
		\item[(2)] if $ q=p_{\alpha}^{*} $, then there exists $ \lambda^{*}>0 $ such that \hyperlink{C}{$ (\mathcal{C}_{\lambda}) $} holds for all $ \lambda\geq \lambda^{*} $. 
	\end{itemize} 
\end{theorem}

Our main result related to non-degenerate case is as follows:
\begin{theorem}\label{Theo1}
Suppose that $ M(0)=m_{0}>0 $. Then, the following statements are hold true: 
\begin{itemize}
	\item[(1)] if $ p<q<p_{\alpha}^{*} $, $ M $ satisfies \hyperlink{M1}{$ (M_{1}) $}-\hyperlink{M2}{$ (M_{2}) $} and $ f $ satisfies \hyperlink{0}{$ (f_{0}) $}, \hyperlink{F1}{$ (F_{1}) $}, \hyperlink{2}{$ (f_{2}) $}, \hyperlink{3}{$ (f_{3}) $}, then \hyperlink{C}{$ (\mathcal{C}_{\lambda}) $} holds for all $ \lambda>0 $.
	\item[(2)] if $ p<q<p_{\alpha}^{*} $, $ M $ satisfies \hyperlink{M1}{$ (M_{1}) $} and $ f $ satisfies \hyperlink{0}{$ (f_{0}) $}, \hyperlink{F1}{$ (F_{1})$}, \hyperlink{F2}{$ (F_{2}) $}, then there exists $ \lambda^{**}>0 $ such that \hyperlink{C}{$ (\mathcal{C}_{\lambda}) $} holds for all $ \lambda\geq \lambda^{**} $.
	\item[(3)] if $ q=p_{\alpha}^{*} $, $ M $ satisfies \hyperlink{M1}{$ (M_{1}) $} and $ f $ satisfies \hyperlink{0}{$ (f_{0}) $}, \hyperlink{2}{$ (F_{1}) $}, \hyperlink{F2}{$ (F_{2}) $}, then there exists $ \lambda^{**}>0 $ such that \hyperlink{C}{$ (\mathcal{C}_{\lambda}) $} holds for all $ \lambda\geq \lambda^{**} $.
\end{itemize}
\end{theorem}

\begin{remark}
	From conclusion \hyperlink{C}{$ (\mathcal{C}_{\lambda}) $}, we have that the energy of any sign-changing solution of \eqref{Equa1} is larger than two times the ground state energy, i.e., the least energy of all weak solutions of \eqref{Equa1}. This property is called energy doubling by Weth in \cite{TW}.
\end{remark}

\textit{-About of method:}

Our strategy is minimization on Nehari sets for energy functional 
\begin{equation*}
	I_{\lambda}(u)=\frac{1}{p}\widehat{M}(\|u\|^{p})-\lambda \int _{\Omega} F(x,u)dx-\frac{1}{q}\int_{\Omega} \frac{|u|^{q}}{|x|^{\alpha}}dx,
\end{equation*}
where $ \widehat{M}(t):=\int_{0}^{t} M(\tau)d\tau $.

We intend to find three solution for the problem \eqref{Equa1} which minimize functional $ I_{\lambda} $ among, respectively, all positive, negative or sign-changing solutions and to check that one of the signed solutions, indeed, minimizes $ I_{\lambda} $ among all non-trivial solutions. Since we are interested in these types of solutions, we consider the following Nehari sets:
\begin{eqnarray*}
	\begin{array}{rl}
	\mathcal{N}_{\lambda}&:=\{u\in W_{0}^{s,p}(\Omega)\setminus \{0\}: \langle I_{\lambda}(u),u \rangle =0\};\\
	\mathcal{N}_{\lambda}^{+}&:=\{u\in \mathcal{N}_{\lambda}: u^{-}=0\};\\
	\mathcal{N}_{\lambda}^{-}&:=\{u\in \mathcal{N}_{\lambda}: u^{+}=0\};\\
	\mathcal{M}_{\lambda}&:=\{u\in W_{0}^{s,p}(\Omega): \langle I_{\lambda}(u),u^{+} \rangle =\langle I_{\lambda}(u),u^{-} \rangle=0, u^{+}\neq 0, u^{-}\neq 0 \},
    \end{array}
\end{eqnarray*}  
where $ u^{+}(x):=\max\{u(x),0\} $ and $ u^{-}(x):=\min\{u(x),0\} $.

More precisely, the idea is to show that the infima
\begin{eqnarray*}
	\begin{array}{c}
		c_{\mathcal{N}_{\lambda}^{+}}:= \inf_{u\in \mathcal{N}_{\lambda}^{+}}I_{\lambda}(u);\\
		c_{\mathcal{N}_{\lambda}^{-}}:= \inf_{u\in \mathcal{N}_{\lambda}^{-}}I_{\lambda}(u);\\
		c_{\mathcal{M}_{\lambda}}:= \inf_{u\in \mathcal{M}_{\lambda}}I_{\lambda}(u)
	\end{array}
\end{eqnarray*}
are achieved by critical points of $ I_{\lambda} $, as well as to show that
 
\begin{equation*}
c_{\mathcal{N}_{\lambda}}:= \inf _{u\in \mathcal{N}_{\lambda}}I_{\lambda}(u)=\min\{c_{\mathcal{\mathcal{N}_{\lambda}^{+}}},c_{\mathcal{N}_{\lambda}^{-}}\}\;\;\operatorname{and}\;\; 	c_{\mathcal{M}_{\lambda}}>	c_{\mathcal{N}_{\lambda}^{+}}+	c_{\mathcal{N}_{\lambda}^{-}}.
\end{equation*}

\textit{Signed solutions:} 

We apply Mountain Pass Theorem to prove that \eqref{Equa1} has ground state solution, namely, that $ c_{\mathcal{N}_{\lambda}} $ is achieved by a critical point of $ I_{\lambda} $. Then, using Ekeland's variational principle, we show that $ c_{\mathcal{N}_{\lambda}^{+}} $ and $ c_{\mathcal{N}_{\lambda}^{-}} $ are also achieved by critical points of $ I_{\lambda} $. Once any ground state solution for problem \eqref{Equa1} has constant sign, we establish that $ c_{\mathcal{N}_{\lambda}}=\min\{c_{\mathcal{N}_{\lambda}^{+}},c_{\mathcal{N}_{\lambda}^{-}}\} $.      

\textit{Sign-changing solution in the subcritical case:} 

In order to find sign-changing solution, the minimization occurs on $ \mathcal{M}_{\lambda} $, however, it seem to be more difficult, even in the subcritical setting. Indeed, if $ u $ changes sign, namely, $ u^{+}\neq 0 $ and $ u^{-}\neq 0 $, then the nonlocal interactions between $ u^{+} $ and $ u^{-} $ result in the inequality
\begin{equation*}
	\|u\|^{p}>\|u^{+}\|^{p}+\|u^{-}\|^{p},
\end{equation*} 
as we will see in \eqref{Equa62}, differently of the local case s =1, on which the equality is valid. This, together with the Kirchhoff term, produce some difficulties even to prove that $ \mathcal{M}_{\lambda} $ is nonempty. In the subcritical setting, we prove that $ c_{\mathcal{M}_{\lambda}} $ is achieved by a critical point of $ I_{\lambda} $ using Brouwer degree in $ \mathbb{R}^{2} $.

\textit{Sign-changing solution in the critical case:}

In critical case, when we try to minimize the functional $ I_{\lambda} $ on $ \mathcal{M}_{\lambda} $ through direct method of Calculus of Variations, we face with a problem due the lack of weak lower semicontinuity of $ I_{\lambda} $. So, we change our strategy by using the solution obtained in the subcritical case to produce the desired solution in the critical case. This can be made through quasi-critical approximation as in \cite{CC,Ta}. We consider a sequence of exponents $ \{q_{n}\}\subset (p,p_{\alpha}^{*}) $ such that $ q_{n}\rightarrow p_{\alpha}^{*} $ and a sequence of least energy sign-changing solutions $ \{u_{n}\} $ for \eqref{Equa1} corresponding to $ \{q_{n}\} $. We prove that, for $ \lambda>0 $ large enough, sequence $ \{u_{n}\} $ converges to a least energy sign-changing solution of \eqref{Equa1} with $ q=p_{\alpha}^{*} $, using a concentration compactness principle with variable exponents developed in \cite{K2}.

The paper is organized as follows. In Section \ref{Pre}, we introduce a variational setting of the problem and present some definitions and preliminary results. In Section \ref{Aux}, we construct a truncation of the function $ M $ and we present a result related to the auxiliary problem associated with the truncated function. In Section \ref{Section2}, we prove some technical results in the degenerate case and we enunciate the corresponding ones in the non-degenerate case. Finally, in Section \ref{Section3}, we prove the main results of this paper.  

\section{Preliminaries}\label{Pre}
\subsection{Definitions and functional spaces}
We start introducing the fractional Sobolev space and some informations about the weak formulation of problem \eqref{Equa1}. For any measurable function $ u:\mathbb{R}^{N}\rightarrow \mathbb{R} $, we define the Gagliardo seminorm by setting
\begin{equation*}
[u]_{s,p}:= \left(\int_{\mathbb{R}^{N}}\frac{|u(x)-u(y)|^{p}}{|x-y|^{N+ps}}dx dy \right)^{\frac{1}{p}}
\end{equation*}
for all $ p\geq 1 $ and the fractional Sobolev space, we define by 
\begin{equation*}
W^{s,p}(\mathbb{R}^{N}):=\{u \in L^{p}(\mathbb{R}^{N}): [u]_{s,p}<\infty\},
\end{equation*}
equipped with the norm
\begin{equation*}
\|u\|_{W^{s,p}(\mathbb{R}^{N})}:=\left(|u|_{p}^{p}+[u]_{s,p}^{p}\right)^{\frac{1}{p}},
\end{equation*}
where $ |\;\cdotp|_{p} $ is the norm in $ L^{p}(\mathbb{R}^{N}) $ (see \cite{D} for more details). Since our problem involves a bounded domain $ \Omega \subset \mathbb{R}^{N} $, we introduce the space
\begin{equation*}
W_{0}^{s,p}(\Omega):=\{u\in W^{s,p}(\mathbb{R}^{N}):\, u(x)=0, \,\operatorname{a.e.}\, x\in \mathbb{R}^{N}\setminus\Omega  \}
\end{equation*}
which is a separable reflexive Banach space with respect to the norm $ [\,\cdot\,]_{s,p} $ and this space can be seen as the completion of $ C_{c}^{\infty}(\mathbb{R}^{N}) $ in the norm $ [\,\cdot\,]_{s,p} $. We denote the topological dual of $ W_{0}^{s,p}(\Omega) $ by $ W^{-s,p'}(\Omega) $, where $ p'=\frac{p}{p-1} $ is the dual exponent of $ p $, and write $ \langle \cdot,\cdot\rangle: W^{-s,p'}(\Omega)\times W_{0}^{s,p}(\Omega)\rightarrow \mathbb{R} $ to designate a duality pairing. In $ W_{0}^{s,p}(\Omega) $, the Gagliardo seminorm is actually a norm, which is equivalent to $ \| \cdot\|_{W^{s,p}(\mathbb{R}^{N})} $ and throughout the text, we will denote by $ \|\cdot\|:=[\,\cdot\,]_{s,p} $.

The Hardy-Sobolev inequality (see \cite[Lemma 2.1]{K2})
\begin{equation*}
\left (\int_{\mathbb{R}^{N}}\frac{|u|^{p_{\alpha}^{*}}}{|x|^{\alpha}}dx\right )^{\frac{1}{p_{\alpha}^{*}}}\leq C(N,p,\alpha)\left (\int _{\mathbb{R}^{2N}}\frac{|u(x)-u(y)|^{p}}{|x-y|^{N+sp}}dxdy\right )^{\frac{1}{p}}
\end{equation*}
implies that the embedding 
\begin{equation*}
W_{0}^{s,p}(\Omega)\hookrightarrow L^{r}\left (\Omega,\frac{dx}{|x|^{\beta}}\right )
\end{equation*}
is continuous, for all $ r\in [1,p_{\beta}^{*}] $ with $ 0\leq \beta<ps $. Furthermore, this embedding is compact for all $ r\in [1,p_{\beta}^{*}) $ (see \cite[Lemma 2.3]{K2}). We denote by
\begin{equation}\label{EquabestC}
S_{\alpha}=\inf _{u\in W_{0}^{s,p}(\Omega)\setminus \{0\}}\frac{\|u\|^{p}}{|u|x|^{-\alpha /p}|_{p_{\alpha}^{*}}^{p}}
\end{equation}
the best constant corresponding to the fractional Hardy-Sobolev embedding when $ r=p_{\alpha}^{*} $.

For each $ u\in W_{0}^{s,p}(\Omega) $, the fractional p-Laplacian operator $ (-\Delta_{p})^{s}u $ can be seen, in weak sense, as a uniquely defined element of $ W^{-s,p'}(\Omega) $ by
\begin{equation*}
\langle (-\Delta_{p})^{s}u,v\rangle=\int_{\mathbb{R}^{2N}}\frac{|u(x)-u(y)|^{p-2}(u(x)-u(y))(v(x)-v(y))}{|x-y|^{N+sp}}dx dy,\;\;\; \forall v\in W_{0}^{s,p}(\Omega). 
\end{equation*}

\begin{definition}
We say that a function $ u\in W_{0}^{s,p}(\Omega) $ is a weak solution of the problem \eqref{Equa1} if
\begin{equation*}
M\left (\|u\|^{p}\right )\langle(-\Delta_{p})^{s}u,v\rangle=\lambda \int_{\Omega}f(x,u)v dx+\int_{\Omega}\frac{|u|^{q-2}u v}{|x|^{\alpha}} dx ,
\end{equation*}
for all $ v \in W_{0}^{s,p}(\Omega) $.
\end{definition}

If $ f $ satisfies \hyperlink{2}{$ (f_{1}) $} or \hyperlink{F1}{$ (F_{1}) $}, then energy functional $ I_{\lambda}:W_{0}^{s,p}(\Omega)\rightarrow \mathbb{R} $ given by
\begin{equation*}
I_{\lambda}(u)= \frac{1}{p}\widehat{M}(\|u\|^{p})- \lambda \int _{\Omega}F(x,u)dx- \frac{1}{q}\int _{\Omega}\frac{|u|^{q}}{|x|^{\alpha}}dx
\end{equation*}
is well defined in $ W_{0}^{s,p}(\Omega) $, with $ \widehat{M}(t)=\int _{0}^{t}M(\tau)d\tau $. Moreover, $ I_{\lambda}\in C^{1}(W_{0}^{s,p}(\Omega),\mathbb{R}) $ and its derivative is given by
\begin{eqnarray*}
\langle I_{\lambda}'(u),v\rangle&=&M\left (\|u\|^{p}\right )\langle (-\Delta_{p})^{s}u,v\rangle-\lambda \int _{\Omega}f(x,u)v dx- \int _{\Omega} \frac{|u|^{q-2}uv}{|x|^{\alpha}} dx,\; \forall v\in W_{0}^{s,p}(\Omega).
\end{eqnarray*}
Consequently, the weak solutions of problem \eqref{Equa1} are the critical points of functional $ I_{\lambda} $. 
\begin{definition}
We say that $ u\in W_{0}^{s,p}(\Omega)\setminus\{0\} $ is a ground state solution of \eqref{Equa1} if $ u $ is a weak solution of \eqref{Equa1} and
\begin{equation*}
I_{\lambda}(u)=\inf \{I_{\lambda}(v): I'_{\lambda}(v)=0, v\neq 0 \},
\end{equation*}
and $ w\in W_{0}^{s,p}(\Omega) $ is a least energy sign-changing solution of \eqref{Equa1} if $ w $ is a weak solution of \eqref{Equa1}, with $ w^{\pm}\neq 0 $ and
\begin{equation*}
I_{\lambda}(w)=\inf \{I_{\lambda}(v): I'_{\lambda}(v)=0, v^{+}\neq 0, v^{-}\neq 0\}.
\end{equation*} 
\end{definition}

\begin{definition}
Let W be a real Banach space and $ J\in C^{1}(W,\mathbb{R}) $. We say that $ \{u_{n}\}\subset W $ is a $ (PS)_{c} $ sequence for $ J $ if $ J(u_{n})\rightarrow c $ and $ J'(u_{n})\rightarrow 0 $, as $ n\rightarrow \infty $. Also, we say that $ J $ satisfies the $ (PS)_{c} $ condition if any $ (PS)_{c} $ sequence has a convergent subsequence. 
\end{definition}

Since we are also interested in to obtain sign-changing solutions, for each $ w\in W_{0}^{s,p}(\Omega) $, we consider $ \Omega_{\pm}=\operatorname{supp} (w^{\pm}) $ and in this way, we can write explicitly the decomposition
\begin{equation*}
\langle (-\Delta_{p})^{s}w,w^{\pm}\rangle = A^{\pm}(w)+B^{\pm}(w),
\end{equation*}
where
\begin{equation*}
A^{\pm}(w):=\int _{\Omega_{\mp}^{c}\times \Omega _{\mp}^{c}}\frac{|w^{\pm}(x)-w^{\pm}(y)|^{p}}{|x-y|^{N+sp}}dxdy,\;\;\;B^{\pm}(w):=2\int _{\Omega _{+}\times \Omega _{-}}\frac{|w^{+}(x)-w^{-}(y)|^{p-1}(w^{\pm}(x)-w^{\pm}(y))}{|x-y|^{N+sp}}dxdy
\end{equation*}
which implies $ \|w\|^{p}=A(w)+B(w) $ with $ A(w):=A^{+}(w)+A^{-}(w) $ and $ B(w):=B^{+}(w)+B^{-}(w) $.

Note that the function $ \phi_{p}(t)=|t|^{p-2}t $ satisfies the inequality 
\begin{equation*}
\phi_{p}(a^{\pm}-b^{\pm})(a^{\pm}-b^{\pm})\leq \phi_{p}(a-b)(a^{\pm}-b^{\pm})\leq \phi_{p}(a-b)(a-b),
\end{equation*}
for all $ a,b\in \mathbb{R} $. Therefore, we have
\begin{equation}\label{Equa62}
\langle (-\Delta_{p})^{s}w^{\pm},w^{\pm}\rangle \leq \langle (\Delta_{p})^{s}w,w^{\pm}\rangle \leq \langle (\Delta_{p})^{s}w,w\rangle,
\end{equation}
for all $ w\in W_{0}^{s,p}(\Omega) $ with strict inequality when $ w $ is sign-changing.

From \hyperlink{M1}{$ (M_{1}) $}, we obtain the following inequality about the Kirchhoff function $ M $:
\begin{equation}
\widehat{M}(t+\theta)\geq \widehat{M}(t)+\widehat{M}(\theta), \;\forall t,\theta \geq 0.
\end{equation}
Then, the energy functional $ I_{\lambda} $ satisfies the inequalities
\begin{equation}\label{Equa55}
I_{\lambda}(w)\geq I_{\lambda}(w^{+})+I_{\lambda}(w^{-}),
\end{equation}
\begin{equation*}
\langle I'_{\lambda}(w),w^{+}\rangle\geq \langle I'_{\lambda}(w^{+}),w^{+}\rangle\; \operatorname{and}\; \langle I'_{\lambda}(w),w^{-}\rangle\geq \langle I'_{\lambda}(w^{-}),w^{-}\rangle
\end{equation*}
for all $ w\in W_{0}^{s,p}(\Omega) $ with strict inequality when $ w $ is sign-changing.

In what follows, whenever it is necessary to emphasize the dependence of exponent $ q $, we will denote the energy functional by $ I_{\lambda,q} $ and the Nehari sets associated with it by $ \mathcal{N}_{\lambda,q} $, $ \mathcal{N}_{\lambda,q}^{+} $, $ \mathcal{N}_{\lambda,q}^{-} $ and $ \mathcal{M}_{\lambda,q} $. Furthermore, whenever a functional $ J:W_{0}^{s,p}(\Omega)\rightarrow \mathbb{R} $ is a bounded from below in $ X\subset W_{0}^{s,p}(\Omega) $, we denote its infimum on $ X $ by
\begin{equation*}
c_{X}:=\inf_{u\in X}J(u).
\end{equation*}

\subsection{Compactness results}
\begin{lemma}\label{Lemm1}
 Let $ \{u_{n}\}\subset W_{0}^{s,p}(\Omega) $ be a bounded sequence with $ u_{n}\rightarrow u $ a.e. $ x\in \Omega $, for some $ u\in W_{0}^{s,p}(\Omega) $.	If $ f $ satisfies \hyperlink{1}{$ (f_{1}) $} or \hyperlink{F1}{$ (F_{1}) $}, then
 \begin{equation*}
 	\int _{\Omega} f(x,u_{n})u_{n}dx \rightarrow \int _{\Omega} f(x,u)u dx\;\;\operatorname{and}\;\; \int _{\Omega} F(x,u_{n})dx\rightarrow \int _{\Omega} F(x,u)dx.
 \end{equation*}  
\end{lemma}
\begin{proof}
	We will prove the result assuming \hyperlink{1}{$ (f_{1}) $}, because the proof is analogous assuming \hyperlink{F1}{$ (F_{1}) $}. Since $ \{u_{n}\} $ is bounded in $ W_{0}^{s,p}(\Omega) $, by \hyperlink{1}{$ (f_{1}) $} we have that $ \{f(x,u_{n})\} $ is bounded in $ L^{p^{*}-1}(\Omega) $ and, thus, there exists $ v\in L^{p^{*}-1}(\Omega) $ such that
	$ f(x,u_{n})\rightharpoonup v $ in $ L^{p^{*}-1}(\Omega) $. Provided that $ f(x,u_{n}) \rightarrow f(x,u) $ a.e. $ x\in \Omega $, it follows that $ v=f(x,u) $ which implies 
	\begin{equation}\label{Eq5}
		\int _{\Omega} f(x,u_{n})u dx \rightarrow \int _{\Omega} f(x,u)u dx.
	\end{equation}
	Given $ \epsilon>0 $, from \hyperlink{2}{$ (f_{1}) $} and H\"older inequality we obtain
\begin{eqnarray*}
	\left |\int _{\Omega}f(x,u_{n})u_{n}-f(x,u)udx\right|&\leq& \int _{\Omega}|f(x,u_{n})||u_{n}-u|dx+\left |\int_{\Omega}(f(x,u_{n})-f(x,u))u dx\right | \\
	&\leq& \epsilon |u_{n}|_{\gamma}|u_{n}-u|_{\gamma}+\epsilon |u_{n}|_{p^{*}}|u_{n}-u|_{p^{*}}+ C_{\epsilon,r}|u_{n}|_{r}|u_{n}-u|_{r}\\
&&	+\int_{\Omega}(f(x,u_{n})-f(x,u))u dx\\
&\leq& C( |u_{n}-u|_{\gamma}+|u_{n}-u|_{r})+C\epsilon + \int_{\Omega}(f(x,u_{n})-f(x,u))u dx.
\end{eqnarray*}
Since $ \epsilon $ is arbitrary, $ \gamma ,r \in (1,p^{*}) $ and \eqref{Eq5} holds, we obtain the first convergence. By similar way, using \hyperlink{1}{$ (f_{1}) $} again and the identity 
\begin{equation*}
	F(x,u_{n})-F(x,u)=\int_{0}^{1}(f(x,tu_{n}+(1-t)u))(u_{n}-u)dt, \; \operatorname{for}\;\operatorname{a.e.} x\in \Omega,
\end{equation*}
one can check the second convergence.  
\end{proof}

Next we will present three results that will be important to get strong convergence of sequences weakly convergent in subsequent results. Their proofs can be found in \cite{K2}. 

\begin{lemma}\label{lem1}
The operator $ (-\Delta_{p})^{s}:W_{0}^{s,p}(\Omega)\rightarrow W^{-s,p'}(\Omega) $ is weak-to-weak continuous.
\end{lemma}

\begin{lemma}\label{lem2}
Let $ \{u_{n}\}\subset W_{0}^{s,p}(\Omega) $ be a bounded sequence with $ u_{n}(x)\rightarrow u(x) $ a.e. $ x\in \Omega $, for some $ u\in W_{0}^{s,p}(\Omega) $ and $ p<q_{n}\leq p_{\alpha}^{*} $ with $ q_{n}\rightarrow p_{\alpha}^{*} $ as $ n\rightarrow \infty $. Then, up to a subsequence, we have
\begin{equation*}
 \int _{\Omega} \frac{|u_{n}|^{q_{n}-2}u_{n}}{|x|^{\alpha}}v dx=\int _{\Omega}\frac{|u|^{p_{\alpha}^{*}-2}u}{|x|^{\alpha}}v dx+o_{n}(1), \;\forall v\in W_{0}^{s,p}(\Omega).
\end{equation*}
\end{lemma}

For $ u\in L_{\operatorname{loc}}^{1}(\Omega) $, define the function
\begin{equation*}
|D^{s}u|^{p}(x):=\int _{\mathbb{R}^{N}}\frac{|u(x)-u(y)|^{p}}{|x-y|^{N+sp}}dy.
\end{equation*}

\begin{lemma}[Concentration compactness principle with variable exponents]\label{lem4}
Let $ 0\leq \alpha\leq ps<N $, $ \Omega $ be a bounded open of $ \mathbb{R}^{N} $ containing $ 0 $ and $ u_{n}\rightharpoonup u $ in $ W_{0}^{s,p}(\Omega) $. Given $ p<q_{n}\leq p_{\alpha}^{*}) $ with $ q_{n}\rightarrow p_{\alpha}^{*} $, there exist two measures $ \nu,\sigma $ and an at most countable set $ \{x_{j}\}_{j\in \mathcal{J}}\subset \overline{\Omega} $ such that, up to subsequence,
\begin{eqnarray}
  |D^{s}u_{n}|^{p}\rightharpoonup^{*}\sigma,\;\;\;\;\frac{|u_{n}|^{q_{n}}}{|x|^{\alpha}}\rightharpoonup^{*}\nu,\label{Eq1}\\
  \sigma \geq |D^{s}u|^{p}+\sum_{j\in \mathcal{J}} \sigma_{j}\delta_{x_{j}}, \;\;\;\ \sigma_{j}=\sigma (\{x_{j}\}),\label{Eq2}\\
  \nu=\frac{|u|^{p_{\alpha}^{*}}}{|x|^{\alpha}}+\sum _{j\in \mathcal{J}}\nu_{j}\delta_{x_{j}}, \;\;\;\nu_{j}=\nu(\{x_{j}\}),\label{Eq3}\\
  \sigma_{j}\geq S_{\alpha}\nu_{j}^{\frac{p}{p_{\alpha}^{*}}}, \;\;\forall j\in \mathcal{J}.\label{Eq4}
\end{eqnarray}
Moreover, if $ \alpha>0 $, then $ \{x_{j}\}_{j\in \mathcal{J}}=\{0\} $.
\end{lemma}

\subsection{Properties of the functions \textit{M} and \textit{f}}

Firstly, note that $ \overline{M} $ to be continuously differentiable in $ [0,\infty) $ means that $ \overline{M} $ has a continuous derivative in $ (0,\infty) $ which can be continuously extended to $ [0,\infty) $. Moreover, as immediate consequence, we have that $ M $ is continuously differentiable in $ (0,\infty) $. 

If $ M $ satisfies \hyperlink{M1}{$ (M_{1}) $} and \hyperlink{M2}{$ (M_{2}) $}, then
\begin{equation}\label{R1}
M(t)\leq M(1)+M(1)t^{\frac{\gamma-p}{p}},\; \forall t\geq 0,
\end{equation}  
\begin{equation}\label{R3}
	M(t)\geq M(1)\min\{1,t^{\frac{\gamma-p}{p}}\}, \; \forall t\geq 0.
\end{equation}
Moreover, we have that
\begin{equation}\label{R4}
	p \overline{M}'(t)t<\gamma\overline{M}(t), \; \forall t>0,
\end{equation}
or written in another way
\begin{equation}\label{R5}
pM'(t)t<(\gamma-p)M(t),\; \forall t>0 
\end{equation}
which implies
\begin{equation}\label{R2}
\frac{1}{p}\widehat{M}(t)-\frac{1}{\gamma}M(t)t \; \operatorname{is}\;\operatorname{increasing}\;\operatorname{and}\;\operatorname{positive}\;\operatorname{for}\; t>0.
\end{equation}

From assumption \hyperlink{0}{$ (f_{0}) $}, it follows that
\begin{equation}\label{propf3}
F(x,t)> 0, 
\end{equation}
for a.e. $ x \in \Omega $ and all $ t\neq 0 $. Also, by \hyperlink{3}{($ f_{3} $)}, we conclude that
\begin{equation}\label{propf1}
\partial _{t} \overline{f}(x,t)t - \gamma f(x,t)t\geq 0, 
\end{equation}
for a.e. $ x \in \Omega $ and all $ t\in \mathbb{R} $, which implies
\begin{equation}\label{propf2}
\frac{1}{\gamma}f(x,t)t-F(x,t)\geq 0 \operatorname{\;and\,it\, is \, increasing\, in \,} t>0 , \operatorname{\, decreasing\, in\,}  t<0.
\end{equation} 
From conditions \hyperlink{1}{$ (f_{1}) $} and \hyperlink{2}{$ (f_{2}) $}, we have that
\begin{equation}\label{Equa2.1}
	F(x,t)\geq C_{1}(x)|t|^{\mu}-C_{2}(x),
\end{equation}
for a.e. $ x\in \Omega $ and all $ t\in \mathbb{R} $, where $ C_{1}(x):=F(x,1)  $ and $ C_{2}(x):=\max_{t\in [0,1]}\left |F(x,t)-C_{1}(x)|t|^{\mu}\right | $ both in $ L^{\infty}(\Omega) $ with $ \mu>\gamma $ given in \hyperlink{2}{$ (f_{2}) $}.

Moreover, if $ f $ satisfies \hyperlink{F1}{($ F_{1} $)} and \hyperlink{F2}{$ (F_{2}) $}, we obtain the following properties:
\begin{equation}
\partial _{t} \overline{f}(x,t)t - \mu f(x,t)t\geq 0, 
\end{equation}
for a.e. $ x \in \Omega $ and all $ t\in \mathbb{R} $  with $ \mu>p $ given in \hyperlink{F2}{$ (F_{2}) $}, which implies
\begin{equation}
\frac{1}{\mu}f(x,t)t-F(x,t)\geq 0 \operatorname{\;and\,it\, is \, increasing\, in \,} t>0 , \operatorname{\, decreasing\, in\,}  t<0.
\end{equation} 

Also, there are non-negative functions $ C_{1},C_{2}\in L^{\infty}(\Omega) $ such that 
\begin{equation}
F(x,t)\geq C_{1}(x)|t|^{\mu}-C_{2}(x),
\end{equation}
for a.e. $ x\in \Omega $ and for all $ t\in \mathbb{R} $.

\section{The auxiliary problem}\label{Aux}

In the case where $ M $ is a Kirchhoff function non-degenerate, namely, $ M(0)=m_{0}>0 $, satisfying only condition \hyperlink{M1}{$ (M_{1}) $}, we shall make a truncation on $ M $ such that the new function obtained $ M_{a} $ still satisfies \hyperlink{M1}{($ M_{1} $)} and satisfies the condition 
\begin{itemize}
	\item[($ M'_{2} $)] \labeltarget{M'2} $ \frac{M(t)}{t^{\frac{\mu-p}{p}}} $ is decreasing in $ t>0 $,
\end{itemize} 
with $ \mu>p $ given in \hyperlink{F2}{$ (F_{2}) $}. This truncation is made following the ideas presented in \cite{NN}.

Suppose that $ M $ satisfies condition \hyperlink{M1}{$ (M_{1}) $} and $ M(0)=m_{0}>0 $. Then, there exists $ t_{0}>0 $ such that
\begin{equation}\label{Equa51}
m_{0}<M(t_{0})<\frac{1}{p}m_{0}\mu,
\end{equation}
for $ \mu >p $ introduced in \hyperlink{F2}{($ F_{2} $)}. Furthermore, since $ \overline{M}' $ is continuous in $ [0,\infty) $ and $ M(0)=m_{0} $, we have
\begin{equation*}
\lim_{t\rightarrow 0^{+}}\overline{M}'(t)=\overline{M}'(0)=m_{0},
\end{equation*}
which implies that there exists $ t_{1}>0 $ such that
\begin{equation*}
p \overline{M}'(t)<\mu m_{0}\leq \mu M(t),\;\;\; \forall t\in [0,t_{1})
\end{equation*}
and, therefore,
\begin{equation}\label{Equa52}
p M'(t)t< (\mu-p)M(t),\;\;\; \forall t\in (0,t_{1}).
\end{equation} 
Let $ t_{2}=\min\{t_{0}, t_{1}\} $ and $ a=M(t_{2}) $. Consider the $ C^{2} $ function
\begin{equation}\label{38}
m(t)=\left \{
\begin{array}{llcc}
t, & t\in [0,\delta]\\
\delta+\frac{2}{\pi}(t_{2}-\delta)\arctan\left(\frac{\pi}{2}\cdot \frac{t-\delta}{t_{2}-\delta}\right), & t\in[\delta,\infty),
\end{array}
\right .
\end{equation}
where $ \delta \in (0,t_{2}) $. Finally, set $ M_{a}:[0,\infty)\rightarrow [0,\infty) $ as
\begin{equation*}
M_{a}(t)=M(m(t)).
\end{equation*} 

From some calculations we infer that 
\begin{equation}\label{Equa54}
 m'(t)>0\;\;  \operatorname{and} \;\; m''(t)\leq 0, \;\; \forall  t>0, 
\end{equation}
which implies
\begin{equation}\label{Equa53}
m(t)\in (0,t_{2})\;\; \operatorname{and} \;\; m'(t)t\leq m(t),\;\; \forall t>0.
\end{equation} 
Then, from \eqref{Equa52}, \eqref{Equa54} and \eqref{Equa53}, it follows that $ M_{a} $ is continuously differentiable in $ (0,\infty) $ and satisfies 
\begin{equation}\label{Equa63}
p M_{a}'(t)t< (\mu-p)M_{a}(t),\;\;\; \forall t>0.
\end{equation}
Thus, $ M_{a} $ will also satisfy \hyperlink{M1}{($ M_{1} $)} and \hyperlink{M'2}{($M'_{2}$)}. Moreover, from \eqref{Equa51}, we have
\begin{equation}\label{Ma1}
m_{0}\leq M_{a}(t)\leq a<\frac{1}{p}m_{0}\mu,\;\forall t\geq 0.
\end{equation}
Also,
\begin{equation}\label{Ma2}
M_{a}(t)=M(t),\; \forall t\in [0,\delta].
\end{equation}
Putting $ \widehat{M}_{a}(t):=\int _{0}^{t} M_{a}(\tau)d\tau $, from \hyperlink{M'2}{$ (M'_{2}) $}, we get
\begin{equation}\label{M'a2}
\frac{1}{p}\widehat{M}_{a}(t)-\frac{1}{\mu}M_{a}(t)t \; \operatorname{is}\;\operatorname{increasing}\;\operatorname{and}\;\operatorname{positive}\;\operatorname{for}\; t>0 .
\end{equation}

From this truncation, we can consider the following auxiliary problem: 

\begin{equation}
\tag{$ P_{\lambda,a} $}
\left \{
\begin{array}{ll}
M_{a}\left (\int _{\mathbb{R}^{2N}}\frac{|u(x)-u(y)|^{p}}{|x-y|^{N+ps}}dx dy\right )(-\Delta_{p})^s u=\lambda f(x,u)+\frac{|u|^{q-2}u}{|x|^{\alpha}} & \operatorname{in}\, \Omega, \\
u=0  &\operatorname{in}\,  \mathbb{R}^{N} \setminus \Omega\\
\end{array}
\right.
\label{Equa40}
\end{equation}
and we denote by $ I_{\lambda,a} $ the energy functional, $ \mathcal{N}_{\lambda,a} $, $ \mathcal{N}_{\lambda,a}^{+} $, $ \mathcal{N}_{\lambda,a}^{-} $ and $ \mathcal{M}_{\lambda,a} $ the corresponding Nehari sets associated with $ I_{\lambda,a} $.

We obtain the following result about the auxiliary problem:
\begin{theorem}\label{Theo21}
	Suppose that $ M(0)=m_{0}>0 $, $ M $ satisfies \hyperlink{M1}{$ (M_{1}) $} and $ f $ satisfies \hyperlink{0}{$ (f_{0}) $}, \hyperlink{F1}{$ (F_{1}) $}, \hyperlink{F2}{$ (F_{2}) $}. Then, the following statements are hold true:
	\begin{itemize}
		\item[(1)] if $ p<q<p_{\alpha}^{*} $, we have that, for all $ \lambda >0 $,
		 \begin{itemize}
		 	\item[$ (\mathcal{C}_{\lambda,a}) $]\labeltarget{C2} problem \eqref{Equa40} has three solutions, $ u_{\lambda,a,1}\in \mathcal{N}_{\lambda,a}^{+} $, $ u_{\lambda,a,2}\in \mathcal{N}_{\lambda,a}^{-} $ and $ u_{\lambda,a,3}\in \mathcal{M}_{\lambda,a} $ such that
		 	\begin{eqnarray*}
		 		I_{\lambda,a}(u_{\lambda,a,1})=c_{\mathcal{N}_{\lambda,a}^{+}},\;\;I_{\lambda,a}(u_{\lambda,a,2})=c_{\mathcal{N}_{\lambda,a}^{-}},\;\; I_{\lambda,a}(u_{\lambda,a,3})=c_{\mathcal{M}_{\lambda,a}}
		 	\end{eqnarray*}
	 	satisfying $ c_{\mathcal{M}_{\lambda,a}}>c_{\mathcal{N}_{\lambda,a}^{+}}+c_{\mathcal{N}_{\lambda,a}^{-}} $ and $ c_{\mathcal{N}_{\lambda,a}}=\min\{ c_{\mathcal{N}_{\lambda,a}^{+}}, c_{\mathcal{N}_{\lambda,a}^{-}}\} $.
		 \end{itemize}
	 	\item[(2)] if $ q=p_{\alpha}^{*} $, then there exists $ \overline{\lambda}>0 $ such that \hyperlink{C2}{$ (\mathcal{C}_{\lambda,a}) $} holds, for all $ \lambda \geq \overline{\lambda} $.
	\end{itemize}
\end{theorem}

\section{Technical results}\label{Section2}

For each $ u\in W_{0}^{s,p}(\Omega)\setminus \{0\} $, we define $ \varphi _{u}:[0,\infty)\rightarrow \mathbb{R} $ by \begin{equation*}
\varphi_{u}(t)=I_{\lambda}(tu).
\end{equation*}

Also, for each $ w\in W_{0}^{s,p}(\Omega) $ with $ w^{+}\neq 0 $ and $ w^{-}\neq 0 $, we consider the function $ \psi_{w}:[0,\infty)\times[0,\infty)\rightarrow \mathbb{R} $ given by
\begin{equation}\label{Equa34}
\psi_{w}(t,\theta)=I_{\lambda}(tw^{+}+\theta w^{-})
\end{equation}
and $ \Psi_{w}:[0,\infty)\times [0,\infty) \rightarrow \mathbb{R}^{2}$ given by
\begin{equation}\label{Equa35}
\Psi _{w}(t,\theta)=\left(t \frac{\partial \psi _{w}}{\partial t}(t,\theta), \theta \frac{\partial \psi _{w}}{\partial \theta}(t,\theta)\right)=
(\langle I'_{\lambda}(tw^{+}+\theta w^{-}),t w^{+}\rangle, \langle I'_{\lambda}(tw^{+}+\theta w^{-}),\theta w^{-}\rangle),
\end{equation}
which is of $ C^{1} $ class, since  $ \overline{f}(x,\cdot) $ is of $ C^{1} $ class, for a.e. $ x\in \Omega $ and \hyperlink{0}{($ f_{0} $)} holds. We denote by $ \varphi_{u,a} $ and $ \psi_{u,a} $ the corresponding functions associate to $ I_{\lambda,a} $.

\subsection{The degenerate case}\label{Subdeg}
In this subsection, we present some previews results considering the degenerate case $ M(0)=0 $. Here, we assume that $ M $ satisfies \hyperlink{M1}{$ (M_{1}) $}-\hyperlink{M2}{$ (M_{2}) $} and $ f $ satisfies \hyperlink{0}{$ (f_{0}) $}-\hyperlink{3}{$ (f_{3}) $}. 

\begin{lemma}\label{Lem1.1}
Functional $ I_{\lambda} $ satisfies the following geometric conditions: 
\begin{itemize}
\item[(1)] For any $ u\in W_{0}^{s,p}(\Omega)\setminus \{0\} $, we have
\begin{equation*}
I_{\lambda}(tu^{+}+\theta u^{-})\rightarrow -\infty, \;\operatorname{as}\;  |(t,\theta)|\rightarrow \infty .
\end{equation*} 
\item[(2)] There exists $ R>0 $ such that 
\begin{equation}
I_{\lambda}(u)\geq R \|u\|^{\gamma} \;\;\operatorname{and}\;\; \langle I'_{\lambda}(u),u\rangle \geq R \|u\|^{\gamma},
\end{equation} 
whenever $ \|u\|\leq R $.
\end{itemize}
\end{lemma}
\begin{proof}
\noindent \textbf{(1)} From \eqref{R1} and \eqref{propf3}, we have
\begin{eqnarray*}
I_{\lambda}(tu^{+}+\theta u^{-})&=&\frac{1}{p}\widehat{M}(\|tu^{+}+\theta u^{-}\|^{p})-\lambda \int _{\Omega} F(x,t u^{+}+\theta u^{-})dx-\frac{1}{q}\int _{\Omega}\frac{|tu^{+}+\theta u^{-}|^{q}}{|x|^{\alpha}}dx\\
&\leq&\frac{M(1)}{p}\|tu^{+}+\theta u^{-}\|^{p}+\frac{M(1)}{\gamma}\|tu^{+}+\theta u^{-}\|^{\gamma}-\frac{1}{q}\int _{\Omega}\frac{|tu^{+}+\theta u^{-}|^{q}}{|x|^{\alpha}}dx\\
&\leq&\frac{2^{p-1}}{p}M(1) t^{p}\|u^{+}\|^{p}+\frac{2^{\gamma-1}}{\gamma}M(1)t^{\gamma}\|u^{+}\|^{\gamma} -\frac{1}{q}t^{q}\int _{\Omega}\frac{|u^{+}|^{q}}{|x|^{\alpha}}dx\\
&&+\frac{2^{p-1}}{p}M(1) \theta ^{p}\|u^{-}\|^{p}+\frac{2^{\gamma-1}}{\gamma}M(1) \theta ^{\gamma}\|u^{-}\|^{\gamma} -\frac{1}{q}\theta^{q}\int _{\Omega}\frac{|u^{-}|^{q}}{|x|^{\alpha}}dx.
\end{eqnarray*}
Since $ p,\gamma<q $, it follows that $ I_{\lambda}(tu^{+}+\theta u^{-})\rightarrow -\infty $, as $ |(t,\theta)|\rightarrow \infty $.

\noindent \textbf{(2)} From \eqref{R3}, assumption \hyperlink{1}{$ (f_{1}) $} and the continuously of embedding $ W_{0}^{s,p}(\Omega)\hookrightarrow L^{r}(\Omega,\frac{dx}{|x|^{\beta}}) $ with $ 0\leq \beta<ps $ and $ r\in[1,p_{\beta}^{*}] $, there exist $ C, C_{\epsilon}, C_{\epsilon,r}, C_{q}>0 $ such that, for $ u\in W_{0}^{s,p}(\Omega) $ with $ \|u\|\leq 1 $, we obtain
\begin{eqnarray*}
I_{\lambda}(u)&=&\frac{1}{p}\widehat{M}(\|u\|^{p})-\lambda \int _{\Omega} F(x,u)dx-\frac{1}{q}\int _{\Omega}\frac{|u|^{q}}{|x|^{\alpha}}dx\\
&\geq& \frac{M(1)}{\gamma}\|u\|^{\gamma} -\frac{\epsilon}{\gamma} \int _{\Omega}|u|^{\gamma}dx-\frac{\epsilon}{p^{*}} \int_{\Omega}|u|^{p^{*}}dx-C_{\epsilon,r}\int _{\Omega}|u|^{r}dx -\frac{1}{q}\int _{\Omega}\frac{|u|^{q}}{|x|^{\alpha}}dx\\
&\geq&\left(\frac{M(1)}{\gamma}-\epsilon C\right)\|u\|^{\gamma} -C_{\epsilon}\|u\|^{p^{*}}-C_{\epsilon,r}\|u\|^{r} -C_{q}\|u\|^{q}.
\end{eqnarray*}
Take $ 0<\epsilon\ll 1 $ such that $ \left(\frac{M(1)}{\gamma}-\epsilon C\right)>0 $. Therefore, since $ p^{*},r,q>\gamma $, there exists $ R>0 $, sufficiently small, such that $ I_{\lambda}(u)\geq R \|u\|^{\gamma} $, whenever $ \|u\|\leq R $. Proceeding in a similar way, one can show the same estimative for $ \langle I_{\lambda}(u),u\rangle $.
\end{proof}

\begin{lemma}\label{Lem1}
For each $ u\in W_{0}^{s,p}(\Omega)\setminus \{0\} $, there exists unique $ t_{u}>0 $ such that $ t_{u}u\in \mathcal{N}_{\lambda} $. In particular, $ \mathcal{N}_{\lambda}\neq \emptyset $ and $ \mathcal{N}_{\lambda}^{\pm}\neq \emptyset $. Moreover, $ I_{\lambda}(t_{u}u)>I_{\lambda}(t u) $, for all $ t\geq 0 $, $ t\neq t_{u} $.
\end{lemma}
\begin{proof}
Fixed $ u\in W_{0}^{s,p}(\Omega)\setminus \{0\} $, from Lemma \ref{Lem1.1}, one can show that there exists $ t_{u}>0 $ such that
\begin{equation*}
\varphi_{u}(t_{u})=\max_{t\geq 0}\varphi_{u}(t)>0.
\end{equation*}
Then $ \varphi_{u}'(t_{u})=0 $ and thus $ t_{u}u\in \mathcal{N}_{\lambda} $. By \hyperlink{M2}{($M_{2}$)} and \hyperlink{3}{$ (f_{3}) $}, it is deduced that $ \frac{\varphi_{u}'(t)}{t^{\gamma-1}} $ is decreasing. Since $ \frac{\varphi_{u}'(t_{u})}{(t_{u})^{\gamma-1}}=0 $, it follows that $ t_{u} $ is the unique point in $ (0,\infty) $ with the property $ \varphi_{u}'(t_{u})=0 $. Finally, by uniqueness, we have $ \varphi_{u}(t_{u})>\varphi_{u}(t) $, for all $ t\geq 0 $ with $ t\neq t_{u} $.
\end{proof}

\begin{corollary}\label{Cor1}
For each $ u\in W_{0}^{s,p}(\Omega)\setminus \{0\} $ with $ \langle I'_{\lambda}(u),u\rangle \leq 0 $,  there exists unique $ t_{u}\in (0,1] $ such that $ t_{u}u \in \mathcal{N}_{\lambda} $.
\end{corollary}
\begin{proof}
From Lemma \ref{Lem1}, there exists unique $ t_{u}>0 $ such that $ t_{u}u\in \mathcal{N}_{\lambda} $. Then, $ \varphi_{u}'(t_{u})=0 $ and by hypothesis $ \varphi_{u}'(1)\leq 0 $. Since $ \frac{\varphi_{u}'(t)}{t^{\gamma-1}} $ is decreasing, we conclude that $ t_{u}\in (0,1] $.
\end{proof}

\begin{lemma}\label{lem3}
For each $ w\in W_{0}^{s,p}(\Omega) $ with $ w^{+}\neq 0$ and $ w^{-}\neq 0 $, there exists unique pair $ (t_{w},\theta_{w})\in (0,\infty)\times (0,\infty) $ such that 
\begin{equation*}
t_{w}w^{+}+\theta _{w} w^{-}\in \mathcal{M}_{\lambda}.
\end{equation*}
In particular, $ \mathcal{M}_{\lambda}\neq \emptyset $. Furthermore, for all $ t,\theta \geq 0 $ with $ (t,\theta)\neq (t_{w},\theta _{w}) $, we have
\begin{equation*}
I_{\lambda}(tw^{+}+\theta w^{-})<I_{\lambda}(t_{w}w^{+}+\theta _{w}w^{-}).
\end{equation*}
\end{lemma}

\begin{proof}
Firstly, we will show the existence of the pair $ (t_{w},\theta _{w}) $. From item (1) of Lemma \ref{Lem1.1}, it follows that
\begin{equation*}
\lim _{|(t,\theta)|\rightarrow \infty}\psi_{w}(t,\theta)\rightarrow -\infty,
\end{equation*} 
which together with continuously of $ \psi _{w} $ imply that there exists $ (t_{w},\theta _{w})\in [0,\infty)\times [0,\infty) $ such that
\begin{equation*}
\psi _{w}(t_{w},\theta _{w})=\max _{(t,\theta)\in [0,\infty)\times [0,\infty)}I_{\lambda}(tw^{+}+\theta w^{-}).
\end{equation*}

Moreover, fixed $ t\geq 0 $, by item (2) of Lemma \ref{Lem1.1}, it follows that
\begin{equation*}
\psi_{w}(t,0)=I_{\lambda}(t w^{+})<I_{\lambda}(tw^{+})+I_{\lambda}(\theta w^{-})\leq I_{\lambda}(tw^{+}+\theta w^{-})=\psi_{w}(t,\theta),
\end{equation*}
whenever $ 0<\theta\ll 1 $. Consequently, $ (t,0) $ is not a maximizer of $ \psi_{w} $, for all $ t\geq 0 $. Analogously, $ (0,\theta ) $ is not a maximizer of $ \psi_{w} $, for all $ \theta\geq 0 $. Thus, we have shown that maximizer $ (t_{w},\theta _{w}) $ is a inner point of $ [0,\infty)\times [0,\infty) $. Then, $ (t_{w},\theta _{w}) $ is a critical point of $ \psi_{w} $ with $ t_{w}>0 $ and $ \theta_{w}>0 $. Therefore, there exists $ (t_{w},\theta_{w})\in (0,\infty)\times (0,\infty) $ such that $ \Psi_{w}(t_{w},\theta_{w})=(0,0) $, namely, $ t_{w}w^{+}+\theta_{w}w^{+}\in \mathcal{M}_{\lambda} $.

Now, we will work to get the uniqueness. In fact, it is sufficient to show that if $ w\in \mathcal{M}_{\lambda} $ and $ tw^{+}+\theta w^{-}\in \mathcal{M}_{\lambda} $ with $ t>0 $ and $ \theta>0 $, then $ (t,\theta)=(1,1) $. Suppose that $ w\in \mathcal{M}_{\lambda} $ and $ tw^{+}+\theta w^{-}\in\mathcal{M}_{\lambda} $. Then, $ \langle I'_{\lambda}(tw^{+}+\theta w^{-}),tw^{+}\rangle=0 $ and $\langle I'_{\lambda}(w),w^{+} \rangle $=0, namely,

\begin{equation*}
M(\|tw^{+}+\theta w^{-}\|^{p})\langle (-\Delta_{p})^{s}(tw^{+}+\theta w^{-}), tw^{+}\rangle =\lambda \int _{\Omega}f(x,t w^{+})tw^{+}dx+t^{q}\int _{\Omega}\frac{|w^{+}|^{q}}{|x|^{\alpha}}dx,
\end{equation*}
\begin{equation*}
M(\|w\|^{p})\left \langle (-\Delta_{p})^{s}w,w^{+}\right \rangle=\lambda \int _{\Omega}f(x,w^{+})w^{+}dx+\int _{\Omega}\frac{|w^{+}|^{q}}{|x|^{\alpha}}dx.
\end{equation*}

Without less of generality, we can suppose that $ \theta\leq t $. Then,
\begin{eqnarray}\label{Equa30}
\|tw^{+}+\theta w^{-}\|^{p}&=&A(t w^{+}+\theta w^{-})+B(t w^{+}+\theta w^{-})\nonumber\\
&\leq& t^{p}(A(w)+B(w))=t^{p}\|w\|^{p}
\end{eqnarray}
and
\begin{eqnarray}\label{Equa31}
\langle (-\Delta_{p})^{s}(t w^{+}+\theta w^{-}),t w^{+}\rangle&=& t^{p}A^{+}(w)+B^{+}(t w^{+}+\theta w^{-})\nonumber\\
&\leq &t^{p}(A^{+}(w)+B^{+}(w))=t^{p}\langle (-\Delta_{p})^{s}w,w^{+}\rangle.
\end{eqnarray}
Since $ M $ is increasing, by \eqref{Equa30} and \eqref{Equa31}, follows that
\begin{eqnarray}\label{Equa32}
\frac{M(t^{p}\|w\|^{p})}{t^{\gamma-p}\|w\|^{\gamma-p}}\langle (-\Delta_{p})^{s}w,w^{+}\rangle\|w\|^{\gamma-p}\geq \lambda \int _{\Omega}\frac{f(x,tw^{+})}{(tw^{+})^{\gamma-1}}(w^{+})^{\gamma}dx+t^{q-\gamma}\int _{\Omega}\frac{|w^{+}|^{q}}{|x|^{\alpha}}dx.
\end{eqnarray}

On the other hand,
\begin{equation}\label{Equa33}
\frac{M(\|w\|^{p})}{\|w\|^{\gamma-p}}\langle (-\Delta_{p})^{s}w,w^{+}\rangle\|w\|^{\gamma-p}=\lambda\int _{\Omega}\frac{f(x,w^{+})}{(w^{+})^{\gamma-1}}(w^{+})^{\gamma}dx+\int _{\Omega}\frac{|w^{+}|^{q}}{|x|^{\alpha}}dx.
\end{equation}

Combining \eqref{Equa32} and \eqref{Equa33}, we get
\begin{eqnarray*}
&&\left (\frac{M(t^{p}\|w\|^{p})}{t^{\gamma-p}\|w\|^{\gamma-p}}-\frac{M(\|w\|^{p})}{\|w\|^{\gamma-p}}\right )\langle (-\Delta_{p})^{s}w,w^{+}\rangle\|w\|^{\gamma-p}\\
&& \;\;\;\;\;\;\;\;\;\;\geq\lambda \int _{\Omega}\left (\frac{f(x,tw^{+})}{(tw^{+})^{\gamma-1}}-\frac{f(x,w^{+})}{(w^{+})^{\gamma-1}}\right )(w^{+})^{\gamma}dx+(t^{q-\gamma}-1)\int _{\Omega}\frac{|w^{+}|^{q}}{|x|^{\alpha}}dx.
\end{eqnarray*}

From \hyperlink{M2}{$ (M_{2}) $}, \hyperlink{3}{$ (f_{3}) $} and the last inequality, we get $ 0<\theta\leq t\leq 1  $. Using the same method, but now with the equations $ \langle I'_{\lambda}(tw^{+}+\theta w^{-}),\theta w^{-} \rangle=0 $ and $ \langle I'_{\lambda}(w),w^{-} \rangle=0 $, we obtain $ 1\leq \theta \leq t $, which implies $ t=\theta=1 $.

Finally, by uniqueness of $ (t_{w},\theta_{w}) $, it follows that 
\begin{equation*}
 I_{\lambda}(tw^{+}+\theta w^{-})<I_{\lambda}(t_{w}w^{+}+\theta _{w}w^{-}),
\end{equation*}
for all $ (t,\theta)\neq (t_{w},\theta_{w}) $. 
\end{proof}

\begin{corollary}\label{Cor2}
For each $ w\in W_{0}^{s,p}(\Omega) $ with $ w^{\pm}\neq 0 $ and $ \langle I'_{\lambda}(w), w^{\pm}\rangle \leq 0$, there exists unique pair $ (t_{w},\theta_{w})\in (0,1]\times (0,1] $ such that $ t_{w}w^{+}+\theta _{w}w^{-}\in \mathcal{M}_{\lambda} $.
\end{corollary}
\begin{proof}
By Lemma \eqref{lem3}, the existence and uniqueness of pair $ (t_{w},\theta _{w})\in (0,\infty)\times (0,\infty) $ are ensured and proceeding as in the proof of that lemma, we obtain $ (t_{w},\theta _{w})\in (0,1]\times (0,1] $.
\end{proof}

\begin{lemma}\label{Lem4}
\noindent (1) There exists $ \kappa:=\kappa_{\lambda}>0 $ such that for all $ u\in \mathcal{N}_{\lambda} $,
\begin{equation*}
\|u\|\geq \kappa 
\end{equation*}
and for all $ w\in \mathcal{M}_{\lambda} $,
\begin{equation*}
\|w^{+}\|, \,\|w^{-}\|\geq \kappa.
\end{equation*}

\noindent (2) For all $ u\in \mathcal{N}_{\lambda} $,
\begin{equation*}
I_{\lambda}(u)\geq \left (\frac{1}{\gamma}-\frac{1}{\mu}\right )M(1)\min\{1,\|u\|^{\gamma-p}\}\|u\|^{p}.
\end{equation*}

\noindent (3) If $ \{u_{n}\}\subset \mathcal{N}_{\lambda} $ and $ \{w_{n}\}\subset \mathcal{M}_{\lambda} $ are bounded sequences, then
\begin{equation*}
\liminf_{n\rightarrow \infty}\int _{\Omega}\frac{|u_{n}|^{q}}{|x|^{\alpha}}dx>0 \;\;\operatorname{and}\;\; \liminf_{n\rightarrow \infty}\int_{\Omega}\frac{|w_{n}^{\pm}|^{q}}{|x|^{\alpha}}dx>0.
\end{equation*}
\end{lemma}
\begin{proof}
	
\noindent \textbf{(1)} If $ w\in \mathcal{M}_{\lambda} $, then $ \langle I'_{\lambda}(w),w^{\pm}\rangle=0 $. Using \eqref{R3}, \hyperlink{M1}{$ (M_{1}) $}, \hyperlink{1}{$ (f_{1}) $} and the H{\"o}lder's inequality, for some constant $ C_{\epsilon,r}>0 $ and $ \|w\|\leq 1 $, it follows that
\begin{eqnarray*}
M(1)\|w^{\pm}\|^{\gamma}&\leq& M(\|w^{\pm}\|^{p})\|w^{\pm}\|^{p}\leq M(\|w\|^{p})\langle (-\Delta_{p})^{s}w,w^{\pm}\rangle=\lambda \int _{\Omega}f(x,w^{\pm})w^{\pm}dx+\int _{\Omega}\frac{|w^{\pm}|^{q}}{|x|^{\alpha}}dx\\
&\leq&\lambda\epsilon \int _{\Omega}(|w^{\pm}|^{\gamma}+|w^{\pm}|^{p^{*}})dx+\lambda C_{\epsilon,r}\int _{\Omega}|w^{\pm}|^{r}dx+\left (\int _{\Omega}\frac{1}{|x|^{\alpha}}dx\right )^{\frac{p_{\alpha}^{*}-q}{p_{\alpha}^{*}}}\left (\int _{\Omega}\frac{|w^{\pm}|^{p_{\alpha}^{*}}}{|x|^{\alpha}}dx\right )^{\frac{q}{p_{\alpha}^{*}}}\\
\end{eqnarray*}
From the continuity of embedding $ W_{0}^{s,p}(\Omega)\hookrightarrow L^{t}(\Omega,\frac{dx}{|x|^{\beta}}) $ with $ 0\leq \beta<ps $ and $ t\in [1,p_{\beta}^{*}] $, we have that
\begin{eqnarray*}
M(1)\|w^{\pm}\|^{\gamma}&\leq&C\lambda\epsilon \|w^{\pm}\|^{\gamma}+\lambda C_{\epsilon}\|w^{\pm}\|^{p^{*}}+\lambda C_{\epsilon,r}\|w^{\pm}\|^{r}+(C_{\alpha})^{\frac{p_{\alpha}^{*}-q}{p_{\alpha}^{*}}}C\|w^{\pm}\|^{q},
\end{eqnarray*}
for some constants $ C,C_{\epsilon}, C_{\epsilon,r}, C_{\alpha}>0 $. Therefore,
\begin{equation}\label{Equa59}
(M(1)-C\lambda\epsilon)\|w^{\pm}\|^{\gamma}\leq  \lambda C_{\epsilon}\|w^{\pm}\|^{p^{*}}+\lambda C_{\epsilon,r}\|w^{\pm}\|^{r}+(C_{\alpha})^{\frac{p_{\alpha}^{*}-q}{p_{\alpha}^{*}}}C\|w^{\pm}\|^{q}.
\end{equation}
Taking $ \epsilon>0 $ such that $ M(1)-C \lambda\epsilon >0 $, we get the conclusion desired, since $ r,q,p^{*}>\gamma $. If $ u\in \mathcal{N}_{\lambda} $, we obtain the same estimative with $ \|u\| $ instead of $ \|w^{\pm}\| $.

\noindent \textbf{(2)} Given $ u\in \mathcal{N}_{\lambda} $, by definition, we have $ \langle I_{\lambda}'(u),u\rangle =0 $. Then, by \hyperlink{3}{$ (f_{2}) $}, \eqref{R2}, \eqref{R3} and since $ \gamma<\mu<q $, we obtain 
\begin{eqnarray*}
I_{\lambda}(u)&=&I_{\lambda}(u)-\frac{1}{\mu}\langle I'_{\lambda}(u),u\rangle\\
&=&\frac{1}{p}\widehat{M}(\|u\|^{p})-\frac{1}{\mu}M(\|u\|^{p})\|u\|^{p}\\
&& +\left (\frac{1}{\mu}-\frac{1}{q}\right )\int _{\Omega} \frac{|u|^{q}}{|x|^{\alpha}}dx + \lambda\int _{\Omega}\frac{1}{\mu}f(x,u)u -F(x,u)dx\\
&\geq&\left (\frac{1}{\gamma}-\frac{1}{\mu}\right )M(1)\min\{1,\|u\|^{\gamma-p}\}\|u\|^{p}.
\end{eqnarray*}

\noindent \textbf{(3)} Again, we will show the property only for $ \{w_{n}\} \subset \mathcal{M}_{\lambda} $. Given $ \epsilon>0 $, using item (1) of this lemma, assumption \hyperlink{1}{$ (f_{1}) $} with $ r\in (p,q) $, the boundedness of $ \{w_{n}\} $ and the H{\"o}lder's inequality, we obtain
\begin{eqnarray*}
M(1)\kappa^{\gamma} \leq M(1)\|w_{n}^{\pm}\|^{\gamma}&\leq& \lambda \epsilon \int _{\Omega}|w_{n}^{\pm}|^{\gamma}dx+\lambda \epsilon\int _{\Omega} |w_{n}^{\pm}|^{p^{*}}dx+\lambda C_{r,\epsilon}\int_{\Omega}|w_{n}^{\pm}|^{r}dx+\int_{\Omega}\frac{|w_{n}^{\pm}|^{q}}{|x|^{\alpha}}dx\\
&\leq&\epsilon C_{\lambda}+C_{\lambda,r,\epsilon}\left (\left (\int _{\Omega}\frac{|w_{n}^{\pm}|^{q}}{|x|^{\alpha}}dx\right )^{\frac{r}{q}}+\int_{\Omega}\frac{|w_{n}^{\pm}|^{q}}{|x|^{\alpha}}dx\right ),
\end{eqnarray*}
which implies 
\begin{equation*}
\liminf_{n\rightarrow \infty}\int _{\Omega}\frac{|w_{n}^{\pm}|^{q}}{|x|^{\alpha}}dx>0.
\end{equation*}
\end{proof}

We use the idea of Tarantello \cite{Tar} to get the following result:
\begin{lemma}\label{lem5}
For each $ u\in \mathcal{N}_{\lambda} $, there exists $ \epsilon>0 $ and a differentiable function $ \xi:B(0,\epsilon)\rightarrow [0,\infty) $ such that $ \xi (0)=1 $, $ \xi(v)(u-v)\in \mathcal{N}_{\lambda} $ for all $ v\in B(0,\epsilon) $ and
\begin{equation*}
\langle \xi'(0),v\rangle= \frac{p \overline{M}'(\|u\|^{p})\langle (-\Delta)^{p}(u),v\rangle-\lambda \int_{\Omega}\partial_{t}\overline{f}(x,u)v dx-q \int _{\Omega}\frac{|u|^{q-2}uv }{|x|^{\alpha}}dx}{p\overline{M}'(\|u\|^{p})\|u\|^{p}-\lambda \int _{\Omega}\partial_{t}\overline{f}(x,u)u-q\int_{\Omega}\frac{|u|^{q}}{|x|^{\alpha}}dx},\;\;\forall v\in W_{0}^{s,p}(\Omega)
.\end{equation*} 
\end{lemma}
\begin{proof}
For each $ u\in \mathcal{N}_{\lambda} $, define $ G_{u}:\mathbb{R}\times W_{0}^{s,p}(\Omega) \rightarrow \mathbb{R} $ by
\begin{equation*}
G_{u}(\xi,w)=\langle I'_{\lambda}(\xi (u-w)),\xi (u-w)\rangle.
\end{equation*}
Then, $ G_{u}(1,0)=\langle I'_{\lambda}(u),u\rangle=0 $ and 
\begin{eqnarray*}
\partial_{t} G_{u}(t,w)\arrowvert_{(1,0)}&=&p\overline{M}'(\|u\|^{p})\|u\|^{p}-\lambda \int _{\Omega}\partial_{t} \overline{f}(x,u)u dx-q \int _{\Omega}\frac{|u|^{q}}{|x|^{\alpha}}dx\\
&\leq& \gamma \overline{M}(\|u\|^{p})-\lambda \int _{\Omega} \partial_{t}\overline{f}(x,u)u dx-q\int _{\Omega}\frac{|u|^{q}}{|x|^{\alpha}}dx\\
&=&\lambda \int _{\Omega}\gamma f(x,u)u-\partial_{t}\overline{f}(x,u)u dx +(\gamma-q)\int _{\Omega} \frac{|u|^{q}}{|x|^{\alpha}}dx<0,
\end{eqnarray*}
by \ref{R4}, \ref{propf1} and since $ \gamma<q $. Applying Implicit Function Theorem, there exists $ \epsilon >0 $  and $ \xi: B(0,\epsilon) \rightarrow [0,\infty) $ differentiable such that 
\begin{equation*}
\langle \xi'(0),v\rangle = -[\partial_{t}G_{u}(1,0)]^{-1}[\partial_{w}G_{u}(1,0)]v, \; \forall v\in W_{0}^{s,p}(\Omega)
\end{equation*}
and
\begin{equation*}
G_{u}(\xi(v),v)=0,
\end{equation*}
for all $ v\in B(0,\epsilon) $, which completes the proof.
\end{proof}

\begin{lemma}\label{Lem6}
For each $ w\in \mathcal{M}_{\lambda} $, we have that $ \det J_{(1,1)}\Psi_{w} >0$, where $ J_{(1,1)}\Psi_{w} $ is the Jacobian matrix of $ \Psi_{w} $ in pair $ (1,1) $.
\end{lemma}

\begin{proof}
Writing 
\begin{equation*}
\Psi_{w}(t,\theta):=(\Psi_{w}^{1}(t,\theta),\Psi_{w}^{2}(t,\theta)),
\end{equation*}
with $\Psi_{w}^{1}(t,\theta)= \langle I_{\lambda}(tw^{+}+\theta w^{-}),tw^{+}\rangle$ and $ \Psi_{w}^{2}(t,\theta)= \langle I_{\lambda}(tw^{+}+\theta w^{-}),\theta w^{-}\rangle  $, across some calculus, it is shows that 
\begin{eqnarray*}
\frac{\partial \Psi_{w}^{1}}{\partial t} (1,1)&=&p M'(\|w\|^{p})(A^{+}(w)+B^{+}(w))^{2}+p M(\|w\|^{p})(A^{+}(w)+B^{+}(w))-M(\|w\|^{p})C(w)\\
&&-\lambda \int _{\Omega}\partial_{t}\overline{f}(x,w^{+})w^{+}dx- q\int _{\Omega}\frac{|w^{+}|^{q}}{|x|^{\alpha}}dx,
\end{eqnarray*} 
\begin{eqnarray*}
\frac{\partial \Psi_{w}^{2}}{\partial \theta} (1,1)&=&pM'(\|w\|^{p})(A^{-}(w)+B^{-}(w))^{2}+p M(\|w\|^{p})(A^{-}(w)+B^{-}(w))-M(\|w\|^{p})C(w)\\
&&-\lambda \int _{\Omega}\partial_{t}\overline{f}(x,w^{-})w^{-}dx- q\int _{\Omega}\frac{|w^{-}|^{q}}{|x|^{\alpha}}dx,
\end{eqnarray*}
\begin{eqnarray*}
\frac{\partial \Psi_{w}^{1}}{\partial \theta} (1,1)&=&\frac{\partial \Psi_{w}^{2}}{\partial t} (1,1)= p M'(\|w\|^{p})(A^{+}(w)+B^{+}(w))(A^{-}(w)+B^{-}(w))+M(\|w\|^{p})C(w)>0,
\end{eqnarray*}
where
\begin{equation*}
C(w)=2(p-1)\int _{\Omega_{+}\times \Omega_{-}}\frac{|w^{+}(x)-w^{-}(y)|^{p-2}w^{+}(x)(-w^{-}(y))}{|x-y|^{N+sp}}dxdy.
\end{equation*}

From \eqref{R5}, it follows that

\begin{eqnarray*}
\frac{\partial \Psi_{w}^{1}}{\partial t} (1,1)&<&\gamma M(\|w\|^{p})(A^{+}(w)+B^{+}(w))\\
&&-pM'(\|w\|^{p})(A^{+}(w)+B^{+}(w))(A^{-}(w)+B^{-}(w))-M(\|w\|^{p})C(w)\\
&&-\lambda \int _{\Omega}\partial_{t}\overline{f}(x,w^{+})w^{+}dx-q\int _{\Omega}\frac{|w^{+}|^{q}}{|x|^{\alpha}}dx.
\end{eqnarray*}

Since $ \langle I'_{\lambda}(w),w^{+}\rangle =0 $, using \eqref{propf1}, we obtain 

\begin{eqnarray*}
\frac{\partial \Psi_{w}^{1}}{\partial t} (1,1)&<&-pM'(\|w\|^{p})(A^{+}(w)+B^{+}(w))(A^{-}(w)+B^{-}(w))-M(\|w\|^{p})C(w) \\
&& +\int _{\Omega}\gamma f(x,w^{+})w^{+}-\partial_{t}\overline{f}(x,w^{+})w^{+}dx+(\gamma-q)\int _{\Omega}\frac{|w^{+}|^{q}}{|x|^{\alpha}}dx\\
&\leq&-pM'(\|w\|^{p})(A^{+}(w)+B^{+}(w))(A^{-}(w)+B^{-}(w))-M(\|w\|^{p})C(w)\\
&=&-\frac{\partial \Psi_{w}^{1}}{\partial \theta} (1,1).
\end{eqnarray*}
Analogously, we can conclude that
\begin{eqnarray*}
\frac{\partial \Psi_{w}^{2}}{\partial \theta} (1,1)<-\frac{\partial \Psi_{w}^{1}}{\partial \theta} (1,1).
\end{eqnarray*}
Therefore,
\begin{eqnarray*}
\det J_{(1,1)}\Psi_{w}&=&\frac{\partial \Psi_{w}^{1}}{\partial t} (1,1) \frac{\partial \Psi_{w}^{2}}{\partial \theta} (1,1)-\frac{\partial \Psi_{w}^{1}}{\partial \theta} (1,1) \frac{\partial \Psi_{w}^{2}}{\partial t} (1,1)\\
&>&\left (\frac{\partial \Psi_{w}^{1}}{\partial \theta} (1,1)\right )^{2}-\left (\frac{\partial \Psi_{w}^{1}}{\partial \theta} (1,1)\right )^{2}=0.
\end{eqnarray*}
\end{proof}

\begin{remark}\label{Rem1}
By Lemma \ref{Lem4}, it follows that $ c_{\mathcal{N}_{\lambda}}>0 $, for all $ \lambda >0 $. Also, given $ w\in \mathcal{M}_{\lambda} $, there exist $ t_{w^{+}}, \theta _{w^{-}}>0 $ such that $ t_{w^{+}}w^{+}, \theta _{w^{-}}w^{-}\in \mathcal{N}_{\lambda} $. So,
\begin{eqnarray*}
I_{\lambda}(w)\geq I_{\lambda}(t_{w^{+}}w^{+}+\theta _{w^{-}}w^{-})>I_{\lambda}(t_{w^{+}}w^{+})+I_{\lambda}(\theta_{w^{-}}w^{-}).
\end{eqnarray*}
Then,
\begin{equation}\label{Equa41}
c_{\mathcal{M}_{\lambda}}\geq c_{\mathcal{N}_{\lambda}^{+}}+c_{\mathcal{N}_{\lambda}^{-}}\geq 2 c_{\mathcal{N}_{\lambda}}.
\end{equation}
Furthermore, if $ c_{\mathcal{M}_{\lambda}} $ is achieved, we have $ c_{\mathcal{M}_{\lambda}}>c_{\mathcal{N}_{\lambda}^{+}}+c_{\mathcal{N}_{\lambda}^{-}}\geq 2c_{\mathcal{N}_{\lambda}} $.
\end{remark}

\begin{proposition}\label{Prop1}
The following asymptotic properties hold: 
\begin{itemize}
\item[(1)] For $ X_{\lambda}=\mathcal{N}_{\lambda},\mathcal{N}_{\lambda}^{+},\mathcal{N}_{\lambda}^{-},\mathcal{M}_{\lambda} $, it holds that $ c_{X_{\lambda}} $ is non-increasing in $ \lambda>0 $ and
\begin{equation*}
\lim_{\lambda\rightarrow \infty} c_{X_{\lambda}} =0.
\end{equation*}
\item[(2)] Let $ \{q_{n}\}\subset (p,p_{\alpha}^{*}] $ be such that $ q_{n}\rightarrow p_{\alpha}^{*} $ as $ n\rightarrow \infty $. Then,
\begin{equation*}
\lim_{\lambda\rightarrow \infty} \limsup_{n\rightarrow \infty}c_{\mathcal{M}_{\lambda,q_{n}}}=0. 
\end{equation*}
\end{itemize}
\end{proposition}

\begin{proof}
\noindent \textbf{(1)} If $ \lambda_{2}>\lambda_{1} $, $ w\in \mathcal{M}_{\lambda_{1}} $ and $ tw^{+}+\theta w^{-}\in \mathcal{M}_{\lambda_{2}} $, by Lemma \ref{lem3}, we have
\begin{eqnarray*}
I_{\lambda_{1}}(w)&=&\psi_{w}^{\lambda_{1}}(1,1)\geq \psi_{w}^{\lambda_{1}}(t,\theta)=\psi_{w}^{\lambda_{2}}(t,\theta)+(\lambda_{2}-\lambda_{1})\int _{\Omega}F(tw^{+}+\theta w^{-})dx \\
&>& \psi_{w}^{\lambda_{2}}(t,\theta)\geq c_{\mathcal{M}_{\lambda_{2}}},
\end{eqnarray*} 
which implies $ c_{\mathcal{M}_{\lambda_{1}}}\geq c_{\mathcal{M}_{\lambda_{2}}} $. Analogously, using $ \varphi_{u} $ instead $ \psi_{w} $, we obtain the same conclusion to other $ c_{X_{\lambda}} $.

Let $ w\in W_{0}^{s,p}(\Omega) $ be with $ w^{+}\neq 0 $ and $ w^{-}\neq 0 $. By Lemma \ref{lem3}, for each $ \lambda >0 $, there exist $ t_{\lambda} >0 $ and $ \theta_{\lambda}> 0$ such that 
\begin{equation*}
t_{\lambda}w^{+}+\theta _{\lambda}w^{-} \in \mathcal{M}_{\lambda}.
\end{equation*}

Now, from \eqref{propf3} and \eqref{R1}, we get
\begin{eqnarray*}
0<c_{X_{\lambda}}\leq c_{\mathcal{M}_{\lambda}}=\inf_{v\in \mathcal{M}_{\lambda}}I_{\lambda}(v)&\leq& I_{\lambda}(t_{\lambda}w^{+}+\theta _{\lambda}w^{-})\\
&\leq& \frac{M(1)}{p} \|t_{\lambda}w^{+}+\theta _{\lambda}w^{-}\|^{p}+\frac{M(1)}{\gamma}\|t_{\lambda}w^{+}+\theta _{\lambda}w^{-}\|^{\gamma},
\end{eqnarray*}
and, thus, it is enough to show that $ t_{\lambda}\rightarrow 0 $ and $ \theta _{\lambda}\rightarrow 0 $, as $ \lambda \rightarrow \infty $.

We consider the set
\begin{equation*}
Q_{w}=\{(t_{\lambda},\theta _{\lambda})\in [0,\infty)\times [0,\infty): \Psi_{w}(t_{\lambda},\theta _{\lambda})=(0,0), \lambda >0\},
\end{equation*}
where $ \Psi_{w} $ was defined in \eqref{Equa35}. Since $ t_{\lambda}w^{+}+\theta _{\lambda}w^{-}\in \mathcal{N}_{\lambda} $, by assumption \hyperlink{0}{($ f_{0} $)} and \eqref{R1}, we have
\begin{eqnarray*}
t_{\lambda}^{q}\int _{\Omega}\frac{|w^{+}|^{q}}{|x|^{\alpha}}dx +\theta _{\lambda}^{q}\int _{\Omega} \frac{|w^{-}|^{q}}{|x|^{\alpha}}dx
&\leq& M(\|t_{\lambda}w^{+}+\theta _{\lambda}w^{-}\|^{p})\|t_{\lambda}w^{+}+\theta _{\lambda}w^{-}\|^{p}\\
&\leq&2^{p-1} M(1) t_{\lambda}^{p}\|w^{+}\|^{p}+2^{p-1} M(1)\theta _{\lambda}^{p}\|w^{-}\|^{p}\\
&&+2^{\gamma-1}M(1)t_{\lambda}^{\gamma}\|w^{+}\|^{\gamma}+2^{\gamma-1}M(1)\theta_{\lambda} ^{\gamma}\|w^{-}\|^{\gamma}.
\end{eqnarray*}

Once $ q >\gamma>p $, it follows that $ Q_{w} $ is bounded. Therefore, if $ \{\lambda _{n}\} \subset (0,\infty)$ is such that $ \lambda _{n}\rightarrow \infty $, as $ n\rightarrow \infty $, up to a subsequence, still denoted by $ \{(t_{\lambda_{n}},\theta_{\lambda_{n}})\} $, there exist $ \overline{t}, \overline{\theta}\geq 0 $ such that $ t_{\lambda _{n}}\rightarrow \overline{t} $ and $ \theta _{\lambda _{n}}\rightarrow \overline{\theta} $.

We will show that $ \overline{t}=\overline{\theta}=0 $. Suppose, by contradiction, that $ \overline{t} >0$ or $ \overline{\theta}> 0$. For each  $ n\in \mathbb{N} $, $ t_{\lambda_{n}}w^{+}+\theta _{\lambda _{n}}w^{-}\in \mathcal{N}_{\lambda_{n}} $, namely
\begin{eqnarray}\label{Equa8}
M(\|t_{\lambda _{n}}w^{+}+\theta _{\lambda _{n}}w^{-}\|^{p})\|t_{\lambda _{n}}w^{+}+\theta _{\lambda _{n}}w^{-}\|^{p}&=&\lambda _{n}\int _{\Omega}f(x,t_{\lambda _{n}}w^{+}+\theta _{\lambda _{n}}w^{-})(t_{\lambda _{n}}w^{+}+\theta _{\lambda _{n}}w^{-})dx\nonumber\\
&& +\int _{\Omega}\frac{|t_{\lambda _{n}}w^{+}+\theta _{\lambda _{n}}w^{-}|^{q}}{|x|^{\alpha}} dx.
\end{eqnarray}
Provided that $ t_{\lambda _{n}}w^{+}\rightarrow \overline{t}w^{+} $ and $ \theta _{\lambda _{n}}w^{-}\rightarrow \overline{\theta}w^{-} $ in $ W_{0}^{s,p}(\Omega) $, by Lemma \ref{Lemm1} and \hyperlink{0}{$ (f_{0}) $}, we have
\begin{eqnarray*}
\int _{\Omega}f(x,t_{\lambda _{n}}w^{+}+\theta _{\lambda _{n}}w^{-})(t_{\lambda _{n}}w^{+}+\theta _{\lambda _{n}}w^{-})dx\rightarrow \int _{\Omega}f(x,\overline{t}w^{+}+\overline{\theta}w^{-})(\overline{t}w^{+}+\overline{\theta}w^{-})dx>0,
\end{eqnarray*}
as $ n\rightarrow \infty $. Once $ \lambda_{n}\rightarrow \infty $, as $ n\rightarrow \infty $ and $\{ t_{\lambda _{n}}w^{+}+\theta _{\lambda _{n}}w^{-} \}$ is bounded in $ W_{0}^{s,p}(\Omega) $, we have a contradiction with equality \eqref{Equa8}. Thus, $ \overline{t}=\overline{\theta}=0 $. Therefore, $ c_{X_{\lambda_{n}}} \rightarrow 0 $, as $ n\rightarrow \infty $.

\noindent \textbf{(2)} Given $ w\in W_{0}^{s,p}(\Omega)  $ with $ w^{+}\neq 0 $ and $ w^{-}\neq 0 $, from Lemma \ref{lem3} we can take $ (t_{\lambda,q_{n}},\theta_{\lambda,q_{n}})\in (0,\infty)\times (0,\infty) $ such that $ t_{\lambda,q_{n}}w^{+}+\theta _{\lambda,q_{n}}w^{-}\in \mathcal{M}_{\lambda,q_{n}} $ and $ (t_{\lambda,p_{\alpha}^{*}},\theta_{\lambda,p_{\alpha}^{*}})\in (0,\infty)\times (0,\infty) $ such that $ t_{\lambda,p_{\alpha}^{*}}w^{+}+\theta _{\lambda,p_{\alpha}^{*}}w^{-}\in \mathcal{M}_{\lambda,p_{\alpha}^{*}} $.

\noindent\emph{Claim:} $ (t_{\lambda,q_{n}},\theta _{\lambda,q_{n}})\rightarrow (t_{\lambda,p_{\alpha}^{*}},\theta _{\lambda,p_{\alpha}^{*}}) $ as $ n\rightarrow \infty $.

Put $ \psi_{w}^{q_{n}}(t,\theta):=I_{\lambda,q_{n}}(tw^{+}+\theta w^{-}) $ and note that, from \eqref{R1} and \eqref{Equa2.1}, we have
\begin{equation*}
\psi_{w}^{q_{n}}(t,\theta)\leq \frac{1}{p}\widehat{M}(\|tw^{+}+\theta w^{-}\|^{p})-\lambda \int _{\Omega}F(x,tw^{+}+\theta w^{-})\rightarrow -\infty
\end{equation*}
as $ |(t,\theta)|\rightarrow \infty $, independently of $ n $. Therefore, there exist $ \tilde{t},\tilde{\theta}>0 $ such that $ (t_{\lambda,q_{n}},\theta _{\lambda,q_{n}})\in (0,\tilde{t}]\times (0,\tilde{\theta}] $, for all $ n $. Then, up to subsequence, $ (t_{\lambda,q_{n}},\theta _{\lambda,q_{n}})\rightarrow (\overline{t},\overline{\theta})\in [0,\infty)\times[0,\infty) $. Since $ \psi_{w}^{q_{n}}\rightarrow \psi_{w}^{p_{\alpha}^{*}} $ locally uniformly as $ n\rightarrow \infty $, it follows that
\begin{equation*}
\psi_{w}^{q_{n}}(t_{\lambda,q_{n}},\theta _{\lambda,q_{n}})\rightarrow \psi_{w}^{p_{\alpha}^{*}}(\overline{t},\overline{\theta})
\end{equation*}
and
\begin{equation*}
\psi_{w}^{q_{n}}(t_{\lambda,q_{n}},\theta _{\lambda,q_{n}})=\sup _{(t,\theta)}\psi_{w}^{q_{n}}(t,\theta)\rightarrow \sup_{(t,\theta)}\psi_{w}^{p_{\alpha}^{*}}(t,\theta)
\end{equation*}
as $ n\rightarrow \infty $. Thus, $ \psi_{w}^{p_{\alpha}^{*}}(\overline{t},\overline{\theta})=\sup _{(t,\theta)}\psi_{w}^{p_{\alpha}^{*}}(t,\theta)=\psi_{w}^{p_{\alpha}^{*}}(t_{\lambda,p_{\alpha}^{*}},\theta _{\lambda,p_{\alpha}^{*}}) $. From the uniqueness, ensured by Lemma \ref{lem3}, we obtain $ t_{\lambda,p_{\alpha}^{*}}=\overline{t} $ and $ \theta _{\lambda,p_{\alpha}^{*}}=\overline{\theta} $. Moreover, from Dominated Convergence Theorem, we obtain
\begin{eqnarray*}
\limsup_{n\rightarrow \infty} c_{\mathcal{M}_{\lambda,q_{n}}}&\leq &\lim_{n\rightarrow \infty}\psi_{w}^{q_{n}}(t_{\lambda,q_{n}},\theta _{\lambda,q_{n}})\\
&=& \psi_{w}^{p_{\alpha}^{*}}(t_{\lambda,p_{\alpha}^{*}},\theta _{\lambda,p_{\alpha}^{*}}).
\end{eqnarray*}
Since $ t_{\lambda,p_{\alpha}^{*}},\theta _{\lambda,p_{\alpha}^{*}}\rightarrow 0 $ as $ \lambda \rightarrow \infty $, it follows that
\begin{equation*}
\lim_{\lambda \rightarrow \infty}\limsup_{n\rightarrow \infty}c_{\mathcal{M}_{\lambda,q_{n}}}=0.
\end{equation*}
\end{proof}

If $ \{u_{n}\} $ is a sequence such that $ u_{n} \rightharpoonup u $ in $ W_{0}^{s,p}(\Omega) $, from a standard argument, one can check that $ u_{n}^{\pm} \rightharpoonup u^{\pm} $. Thus, we can apply Lemma \ref{lem4} to $ \{u_{{n}}\} $ and to both $ u_{{n}}^{\pm} $, so that, for $ u $, we obtain measures $ \nu, \sigma $, $ \{x_{j}\}_{j\in \mathcal{J}}\subset \overline{\Omega} $ satisfying \eqref{Eq1}-\eqref{Eq4} and for $ u^{\pm} $, correspondents $ \nu^{\pm} $, $ \sigma^{\pm} $, $ \{x_{j}^{\pm}\}_{j\in \mathcal{J}^{\pm}} $.

\begin{lemma}\label{lem6}
Let $ \{u_{n}\} $ and $ \{q_{n}\} $ be a sequences satisfying the hypotheses of Lemma \ref{lem4} such that $ I'_{\lambda,q_{n}}(u_{n})\rightarrow 0 $, as $ n\rightarrow \infty $. With the previous notations, if $ \nu_{j}>0 $ for some $ j\in \mathcal{J} $ or $  \nu _{k}^{\pm}>0 $ for some $ k\in \mathcal{J}^{\pm} $, we have
\begin{equation*}
\nu_{j}, \nu _{k}^{\pm}\geq \min\left \{\left (M(1)S_{\alpha}\right )^{\frac{N-\alpha}{sp-\alpha}},\left (M(1)S_{\alpha}^{\frac{\gamma}{p}}\right )^{\frac{p_{\alpha}^{*}}{p_{\alpha}^{*}-\gamma}}\right \}.
\end{equation*}
\end{lemma}
\begin{proof}
Fix $ k\in \mathcal{J^{\pm}} $ such that $ \nu_{j}^{\pm}>0 $, $ x_{j}^{\pm}\in \overline{\Omega} $ and for $ \varrho >0 $ consider $ \phi_{\varrho}\in C_{c}^{\infty}(B(x_{j},2\varrho)) $ such that
\begin{equation*}
0\leq \phi_{\varrho}\leq 1,\;\; \phi\arrowvert_{B(x_{j},\varrho)}=1,\;\;|\nabla\phi_{\varrho}|\leq \frac{C}{\varrho}.
\end{equation*}
Since $ \{\phi_{\varrho}u_{n}^{\pm}\} $ is bounded, we have that $ \langle I'_{\lambda,q_{n}}(u_{n}),\phi_{\varrho}u_{n}^{\pm}\rangle=o_{n}(1) $ and thus
\begin{eqnarray*}
\int _{\Omega}\frac{|u_{n}^{\pm}|^{q_{n}}}{|x|^{\alpha}}\phi_{\varrho}dx+\lambda\int _{\Omega}f(x,u_{n}^{\pm})u_{n}^{\pm}\phi _{\varrho}dx
=M(\|u_{n}\|^{p})\langle (-\Delta_{p})^{s}u_{n},\phi _{\varrho}u_{n}^{\pm}\rangle +o_{n}(1).
\end{eqnarray*}
Note that, for all $ n $
\begin{eqnarray*}
\langle (-\Delta_{p})^{s}u_{n},\phi _{\varrho}u_{n}^{\pm}\rangle &=&\int _{\mathbb{R}^{2N}}\frac{|u_{n}(x)-u_{n}(y)|^{p-2}(u_{n}(x)-u_{n}(y))(\phi _{\varrho}(x)u_{n}^{\pm}(x)-\phi _{\varrho}(y)u_{n}^{\pm}(y))}{|x-y|^{N+sp}}dxdy\\
&\geq&\int _{\mathbb{R}^{N}}|D^{s}u_{n}^{\pm}|^{p}\phi _{\varrho}dx-\left |\int _{\mathbb{R}^{2N}}\frac{|u_{n}(x)-u_{n}(y)|^{p-2}(u_{n}(x)-u_{n}(y))u_{n}^{\pm}(y)(\phi _{\varrho}(x)-\phi _{\varrho}(y))}{|x-y|^{N+sp}}dxdy\right |.
\end{eqnarray*}
Proceeding as in \cite[Lemma 3.1]{Mo}, we obtain
\begin{eqnarray*}
\limsup_{n\rightarrow \infty}\left |\int _{\mathbb{R}^{2N}}\frac{|u_{n}(x)-u_{n}(y)|^{p-2}(u_{n}(x)-u_{n}(y))u_{n}^{\pm}(y)(\phi _{\varrho}(x)-\phi _{\varrho}(y))}{|x-y|^{N+sp}}dxdy\right |\leq C\left (\int _{\mathbb{R}^{N}}|D^{s}\phi _{\varrho}|^{p}|u^{\pm}|^{p}dy\right )^{\frac{1}{p}}.
\end{eqnarray*}
Then, putting $ \|u_{n}^{\pm}\|\rightarrow \tau_{\pm}\geq 0 $ as $ n\rightarrow \infty $, we can take the limit in $ n $ to conclude that
\begin{eqnarray*}
M(\tau_{\pm}^{p})\int _{\mathbb{R}^{N}}\phi _{\varrho}d\sigma^{\pm}\leq\int _{\mathbb{R}^{N}}\phi _{\varrho}d\nu^{\pm} +\lambda \int _{\Omega} f(x,u^{\pm})u^{\pm} \phi _{\varrho}dx+C M(\tau_{\pm}^{p}) \left (\int _{\mathbb{R}^{N}}|D^{s}\phi _{\varrho}|^{p}|u^{\pm}|^{p}dy\right )^{\frac{1}{p}}.
\end{eqnarray*}
Moreover, by \cite[(2.14)]{Mo}, it follows that
\begin{equation*}
\lim_{\varrho\rightarrow 0}\int _{\mathbb{R}^{N}}|D^{s}\phi _{\varrho}|^{p}|u|^{p}dy=0
\end{equation*}
which implies in the inequality
\begin{equation*}
\nu_{k}^{\pm}\geq M(\tau_{\pm}^{p})\sigma_{k}^{\pm}.
\end{equation*}

Since $ |D^{s}u_{n}^{\pm}|^{p}\rightharpoonup^{*}\sigma ^{\pm} $, from \eqref{Eq2}, we have that $ \tau _{\pm}^{p}\geq \sigma_{k}^{\pm}  $. Then, by \eqref{R3}, it follows that 
\begin{equation*}
	\nu_{k}^{\pm}\geq M(1)\min\{1,(\sigma_{k}^{\pm})^{\frac{\gamma-p}{p}}\}\sigma_{k}^{\pm}.
\end{equation*}
From \eqref{Eq4}, we have that
\begin{equation*}
\nu_{k}^{\pm}\geq \min\left \{\left (M(1)S_{\alpha}\right )^{\frac{N-\alpha}{sp-\alpha}},\left (M(1)S_{\alpha}^{\frac{\gamma}{p}}\right )^{\frac{p_{\alpha}^{*}}{p_{\alpha}^{*}-\gamma}}\right \}.
\end{equation*}
By similar arguments, the same inequality can be obtained for $ \nu_{j} $.
\end{proof}

\begin{proposition}[PS condition]\label{Prop3}
The following statements are hold true:

\noindent (1) for $ p<q<p_{\alpha}^{*} $, $ I_{\lambda} $ satisfies $ (PS)_{c} $, for all $ c\in \mathbb{R} $;

\noindent (2) for $ q=p_{\alpha}^{*} $, $ I_{\lambda} $ satisfies $ (PS)_{c} $, for all
\begin{equation*}
c<\left (\frac{1}{\gamma}-\frac{1}{p_{\alpha}^{*}}\right )\min\left \{\left (M(1)S_{\alpha}\right )^{\frac{N-\alpha}{sp-\alpha}},\left (M(1)S_{\alpha}^{\frac{\gamma}{p}}\right )^{\frac{p_{\alpha}^{*}}{p_{\alpha}^{*}-\gamma}}\right \}.
\end{equation*}
\end{proposition}
\begin{proof}

\noindent \textbf{(1)} Let $ \{u_{n}\} $ be a $ (PS)_{c} $ sequence of $ I_{\lambda} $, in other words,
\begin{equation*}
I_{\lambda}(u_{n})\rightarrow c, \;\; I'_{\lambda}(u_{n})\rightarrow 0\; \operatorname{in}\;W^{-s,p}(\Omega), \; n\rightarrow \infty.
\end{equation*}

Then, by \hyperlink{2}{$ (f_{2}) $} and \eqref{R2}, since $\gamma<\mu<q $, we obtain
\begin{eqnarray*}
c+o_{n}(1)\|u_{n}\|&=&I_{\lambda}(u_{n})-\frac{1}{\mu}\langle I'_{\lambda}(u_{n}),u_{n}\rangle\\
&=&\frac{1}{p}\widehat{M}(\|u_{n}\|^{p})-\frac{1}{\mu}M(\|u_{n}\|^{p})\|u_{n}\|^{p}+\lambda \int _{\Omega}\frac{1}{\mu}f(x,u_{n})u_{n}-F(x,u_{n})dx\\
&&+\left (\frac{1}{\mu}-\frac{1}{q}\right )\int _{\Omega}\frac{|u_{n}|^{q}}{|x|^{\alpha}}dx\geq \left (\frac{1}{\gamma}-\frac{1}{\mu}\right )M(\|u_{n}\|^{p})\|u_{n}\|^{p}.
\end{eqnarray*}

\noindent \emph{Case 1:} $ \inf_{n}\|u_{n}\|=d_{\lambda}>0  $.

From \hyperlink{M1}{$ (M_{1}) $}, we have that $ M(\|u_{n}\|^{p})\geq M(d_{\lambda}^{p})>0 $ which implies in the boundedness of $ \{u_{n}\} $ since
\begin{equation*}
	c+o_{n}(1)\|u_{n}\|\geq \left (\frac{1}{\gamma}-\frac{1}{\mu}\right )M(d_{\lambda}^{p})\|u_{n}\|^{p}.
\end{equation*}
Therefore, up to a subsequence, $ u_{n}\rightharpoonup u\in W_{0}^{s,p}(\Omega) $ and since the embedding $ W_{0}^{s,p}(\Omega) \hookrightarrow L^{r}(\Omega,\frac{dx}{|x|^{\beta}}) $ is compact for each $ r\in [1,p_{\beta}^{*}) $ with $ 0\leq\beta<ps $, up to subsequence, we can suppose that
\begin{equation*}
\begin{array}{lllc}
u_{n}\rightarrow u,\, \operatorname{in}\, L^{r}(\Omega,\frac{dx}{|x|^{\beta}}),\\
u_{n}(x)\rightarrow u(x), \, \operatorname{a.e.} \, x\in \Omega,\\
\|u_{n}\|\rightarrow \tau_{\lambda}> 0.
\end{array}
 \end{equation*}
Using Lemma \ref{Lemm1}, the weak convergence of $\{ u_{n}\}$ and Lemma \ref{lem2}, we obtain
\begin{equation}\label{Equa47}
\int_{\Omega}f(x,u_{n})(u_{n}-u)dx\rightarrow 0\;\;\operatorname{and}\;\; \int_{\Omega}\frac{|u_{n}|^{q-2}u_{n}(u_{n}-u)}{|x|^{\alpha}}dx\rightarrow 0, \; n\rightarrow \infty.
\end{equation} 
On the other hand, $ \langle I'_{\lambda}(u_{n}),u_{n}-u \rangle=o_{n}(1) $, which together with \eqref{Equa47} imply
\begin{equation*}
M(\tau_{\lambda})\langle (-\Delta_{p})^{s}u_{n},u_{n}-u\rangle =o_{n}(1).
\end{equation*}
Since $ M(\tau_{\lambda})>0 $, we conclude that
\begin{equation*}
\langle (-\Delta_{p})^{s}u_{n},u_{n}-u\rangle=o_{n}(1),
\end{equation*}
which, by Lemma \ref{lem1}, implies $ \|u_{n}\|\rightarrow \|u\| $ as $ n\rightarrow \infty $ and since $ W_{0}^{s,p}(\Omega) $ is uniformly convex it follows that $ \{u_{n}\} $ converges strongly to $ u $.

\noindent \emph{Caso 2:} $  \inf _{n}\|u_{n}\|=0 $

In this case, either $ 0 $ is an accumulation point of the sequence $ \{\|u_{n}\|\} $ or $ 0 $ is a isolated point of $ \{\|u_{n}\|\} $. If the first possibility occurs, we have that $ u_{n}\rightarrow 0 $ strongly and, consequently, the result is proved. If the second one occurs, there exists a subsequence $ \{\|u_{n_{k}}\|\} $ such that $ \inf_{k}\|u_{n_{k}}\|=d_{\lambda}>0 $ and we can proceed as in case 1.  

\noindent \textbf{(2)} Let $ \{u_{n}\} $ be a $ (PS)_{c} $ sequence of $ I_{\lambda} $ with
\begin{equation*}
c<\left (\frac{1}{\gamma}-\frac{1}{p_{\alpha}^{*}}\right )\min\left \{\left (M(1)S_{\alpha}\right )^{\frac{N-\alpha}{sp-\alpha}},\left (M(1)S_{\alpha}^{\frac{\gamma}{p}}\right )^{\frac{p_{\alpha}^{*}}{p_{\alpha}^{*}-\gamma}}\right \}.
\end{equation*}
As in (1), firstly, we address the case $ \inf_{n}\|u_{n}\|=d_{\lambda}>0 $. Then, by similar way, we have that $ \{u_{n}\} $ is bounded and, up to subsequence, $ u_{n} \rightharpoonup u $ in $ W_{0}^{s,p}(\Omega) $. Applying the concentration compactness principle (see Lemma \ref{lem4}) to sequence $ \{u_{n}\} $ with $ q_{n}=p_{\alpha}^{*} $, there exist two measures $ \nu $, $ \sigma $ and an at most countable set $ \{x_{j}\}_{j\in \mathcal{J}} $ satisfying \eqref{Eq1}-\eqref{Eq4}. Supposing, by contradiction, that there exists $ j\in \mathcal{J} $ such that $ \nu_{j}>0 $, we can apply Lemma \ref{lem6} to obtain
\begin{eqnarray*}
c&=&I_{\lambda}(u_{n})-\frac{1}{\gamma}\langle I'_{\lambda}(u_{n}),u_{n}\rangle +o_{n}(1)\\
&\geq &\left (\frac{1}{\gamma}-\frac{1}{p_{\alpha}^{*}}\right )\int _{\Omega}\frac{|u_{n}|^{p_{\alpha}^{*}}}{|x|^{\alpha}}dx+o_{n}(1)\\
&\geq& \left (\frac{1}{\gamma}-\frac{1}{p_{\alpha}^{*}}\right )\nu_{j} \geq \left (\frac{1}{\gamma}-\frac{1}{p_{\alpha}^{*}}\right )\min\left \{\left (M(1)S_{\alpha}\right )^{\frac{N-\alpha}{sp-\alpha}},\left (M(1)S_{\alpha}^{\frac{\gamma}{p}}\right )^{\frac{p_{\alpha}^{*}}{p_{\alpha}^{*}-\gamma}}\right \},
\end{eqnarray*}
which is a contradiction. Therefore, $ \mathcal{J}=\emptyset $ and $ u_{n}\rightarrow u $ in $ L^{p_{\alpha}^{*}}(\Omega,\frac{dx}{|x|^{\alpha}}) $. By standard arguments $ u_{n}\rightarrow u $ in $ W_{0}^{s,p}(\Omega) $. Finally, if $ \inf_{n}\|u_{n}\|=0 $, we can proceed as in case 2 of item (1).
\end{proof}

\begin{lemma}\label{Lem7}
	There exists a closed set $ V^{\pm} $ with $ \mathcal{N}_{\lambda}^{\pm}\subset V^{\pm}\subset \mathcal{N}_{\lambda} $ and sequences $ \{v_{n}\}\subset V^{\pm} $, $ \{u_{n}\}\subset \mathcal{N}_{\lambda}^{\pm} $ such that $ \{v_{n}\} $ is a $ (PS)_{c_{V^{\pm}}} $ sequence with
\begin{equation}\label{Equa64}
	\lim_{n\rightarrow \infty}\|u_{n}-v_{n}\|=0,
\end{equation}
where
\begin{equation*}
	c_{V^{\pm}}:=\inf _{v\in V^{\pm}}I_{\lambda}(v).
\end{equation*}
\end{lemma}
\begin{proof}
	We will make the proof only for $ \mathcal{N}_{\lambda}^{+} $, because the proof for $ \mathcal{N}_{\lambda}^{-} $ is similar. Let $ R>0 $ be given in Lemma \ref{Lem1.1} and consider the open set 
	\begin{equation*}
	U^{+}:=\{u\in W_{0}^{s,p}(\Omega): \|u^{-}\|<R\}
	\end{equation*} 
	and the closed set containing $\mathcal{N}_{\lambda}^{+} $
	\begin{equation*}
	V^{+}:=\mathcal{N}_{\lambda}\cap \overline{U^{+}}.
	\end{equation*}
	
	Let $ \{\tilde{u}_{n}\}\subset V^{+} $ be a minimizing sequence for $ I_{\lambda} $ in $ V^{+} $, namely, $ I_{\lambda}(\tilde{u}_{n})\rightarrow c_{V^{+}} $, as $ n\rightarrow \infty $. Applying the first inequality of item (2) of Lemma \ref{Lem1.1} to $ \tilde{u}_{n}^{-} $, we obtain
	\begin{eqnarray*}
		\langle I'_{\lambda}(\tilde{u}_{n}^{+}),\tilde{u}_{n}^{+}\rangle& \leq & \langle I'_{\lambda}(\tilde{u}_{n}^{+}),\tilde{u}_{n}^{+}\rangle+\langle I'_{\lambda}(\tilde{u}_{n}^{-}),\tilde{u}_{n}^{-}\rangle\\
		&\leq& \langle I'_{\lambda}(\tilde{u}_{n}), \tilde{u}_{n} \rangle=0
	\end{eqnarray*}
	and, thus, by Corollary \ref{Cor1}, there exists $ t_{n}\in (0,1] $ such that $ u_{n}:=t_{n}\tilde{u}_{n}^{+}\in \mathcal{N}_{\lambda}^{+} $. Furthermore, applying the second inequality of item (2) of Lemma \ref{Lem1.1} to $ t_{n}\tilde{u}_{n}^{-} $ and using Lemma \ref{Lem1}, we have
	\begin{equation*}
	I_{\lambda}(u_{n})\leq I_{\lambda}(u_{n})+I_{\lambda}(t_{n}\tilde{u}_{n}^{-})\leq I_{\lambda}(t_{n}\tilde{u}_{n})\leq I_{\lambda}(\tilde{u}_{n}).
	\end{equation*} 
	Therefore, $ \{u_{n}\}\subset \mathcal{N}_{\lambda}^{+} $ is also a minimizing sequence for $ I_{\lambda} $ in $ V^{+} $. Applying Ekeland's variational principle (see \cite[Theorem 1.1]{EK}), there exists $ \{v_{n}\}\subset V^{+} $ such that
	\begin{equation}\label{Equa43}
	\|u_{n}-v_{n}\|\leq \frac{1}{\sqrt{n}}, \;\; I_{\lambda}(v_{n})\leq I_{\lambda}(u_{n})<c_{V^{+}}+\frac{1}{n}
	\end{equation} 
	and
	\begin{equation}\label{Equa44}
	I_{\lambda}(v_{n})<I_{\lambda}(v)+\frac{1}{\sqrt{n}}\|v_{n}-v\|,
	\end{equation}
	for all $ v\in V^{+} $, $ v\neq v_{n} $. 
	
	Note that $ v_{n}\in U^{+} $ if $ n $ is sufficiently large. In fact, by Corollary \ref{Cor1}, there exists $ \theta_{n}\in (0,1] $ such that $ \theta_{n}v_{n}^{+}\in \mathcal{N}_{\lambda}^{+} $. Since $ \{ v_{n}\} $ is bounded, there exists $ C>0 $ such that $ \|v_{n}^{+}\|^{\gamma}\leq\|v_{n}\|^{\gamma}\leq C $ and, thus, by Lemma \ref{Lem4},
	\begin{equation*}
	\theta _{n}^{\gamma}\geq \frac{\kappa ^{\gamma}}{C},
	\end{equation*}
	for all $ n\in \mathbb{N} $. Then, set $ N:=\{n\in \mathbb{N}: \|v_{n}^{-}\|=R\} $ is finite, because otherwise, using again item (2) of Lemma \ref{Lem1.1}, we get
	\begin{equation*}
	I_{\lambda}(\theta_{n}v_{n}^{-})\geq R \theta _{n}^{\gamma}\|v_{n}^{-}\|^{\gamma} \geq R^{\gamma+1}\frac{\kappa^{\gamma}}{C}>\frac{1}{n}
	\end{equation*}
	which implies
	\begin{eqnarray*}
		I_{\lambda}(v_{n}) \geq  I_{\lambda}(\theta_{n}v_{n})\geq I_{\lambda}(\theta_{n}v_{n}^{+})+I_{\lambda}(\theta_{n}v_{n}^{-})
		\geq c_{V^{+}}+\frac{1}{n},
	\end{eqnarray*}
	if $ n\in N $, $ n $ sufficiently large, contradicting \eqref{Equa43}.
	
	Since $ U^{+} $ is open, there exists $ \delta_{n}>0 $ such that $ B(v_{n},\delta_{n})\subset U^{+} $, for all $ n $ sufficiently large. From Lemma \ref{lem5}, we obtain the functions $ \xi_{n}:B(0,\epsilon_{n})\rightarrow [0,\infty) $ such that, by continuously, $ \epsilon_{n}>0 $ can be chosen satisfying 
	\begin{equation*}
	|\xi_{n}(v)-\xi_{n}(0)|<\frac{\delta_{n}}{2C}, \;\;\forall v \in B(0,\epsilon_{n}), 
	\end{equation*}
	with $ \epsilon_{n}<\frac{\delta_{n}C}{\delta_{n}+2C} $. Then, since $ \xi_{n}(0)=1 $, we have that $ \xi_{n}(v)(v_{n}-v)\in B(v_{n},\delta_{n}) \subset U^{+}$, for all $ v\in B(0,\epsilon_{n}) $. 
	
	Let $ 0\leq \rho <\epsilon_{n} $, $ v\in W_{0}^{s,p}(\Omega)\setminus \{0\} $, $ v_{\rho}:=\frac{\rho v}{\|v\|} $ and $ \eta_{\rho}:=\xi_{n}(v_{\rho})(v_{n}-v_{\rho}) $. Since $ \eta_{\rho}\in \mathcal{N}_{\lambda}\cap U^{+} $, by \eqref{Equa44}, it follows that
	\begin{equation*}
	I_{\lambda}(\eta_{\rho})-I_{\lambda}(v_{n})\geq -\frac{1}{\sqrt{n}}\|\eta _{\rho}-v_{n}\|.
	\end{equation*}
	From definition of Fr\'echet derivative, we obtain
	\begin{equation*}
	\langle I'_{\lambda}(v_{n}),\eta _{\rho}-v_{n}\rangle +o_{\rho}(\|\eta _{\rho}-v_{n}\|)\geq -\frac{1}{\sqrt{n}}\|\eta _{\rho}-v_{n}\|.
	\end{equation*}
	Thus, 
	\begin{equation*}
	\langle I'_{\lambda}(v_{n}),-v_{\rho}\rangle +(\xi_{n}(v_{\rho})-1)\langle I'_{\lambda}(v_{n}),v_{n}-v_{\rho}\rangle \geq -\frac{1}{\sqrt{n}}\|v_{\rho}-v_{n}\|+o_{\rho}(\|\eta_{\rho}-v_{n}\|).
	\end{equation*}
	
	Since $ \xi_{n}(v_{\rho})(v_{n}-v_{\rho})\in \mathcal{N}_{\lambda} $, we have that
	\begin{eqnarray*}
		&&-\rho \left \langle I'_{\lambda}(v_{n}),\frac{v}{\|v\|}\right \rangle +(\xi_{n}(v_{\rho})-1)\langle I'_{\lambda}(v_{n})-I'_{\lambda}(\eta _{\rho}),v_{n}-v_{\rho}\rangle\\
		&&\;\;\;\;\geq -\frac{1}{\sqrt{n}}\|\eta _{\rho}-v_{n}\|+o_{\rho}(\|\eta _{\rho}-v_{n}\|).
	\end{eqnarray*}
	
	Thus,
	\begin{eqnarray*}
		\left \langle I'_{\lambda}(v_{n}),\frac{v}{\|v\|}\right \rangle &\leq& \frac{1}{\sqrt{n}}\frac{\|\eta _{\rho}-v_{n}\|}{\rho}+\frac{o_{\rho}(\|\eta _{\rho}-v_{n}\|)}{\rho}\\
		&&+\frac{(\xi_{n}(v_{\rho})-1)}{\rho}\langle I'_{\lambda}(v_{n})-I'_{\lambda}(\eta_{\rho}),v_{n}-v_{\rho}\rangle.
	\end{eqnarray*}
	
	From inequalities
	\begin{equation*}
	\|\eta_{\rho}-v_{n}\|\leq |\xi_{n}(v_{\rho})-1|\|v_{n}\|+|\xi_{n}(v_{\rho})|\rho
	\end{equation*}
	and
	\begin{equation*}
	\lim_{\rho\rightarrow 0}\frac{|\xi_{n}(v_{\rho})-1|}{\rho}\leq \|\xi'_{n}(0)\|,
	\end{equation*}
	for $ \rho \rightarrow 0 $ and $ n $ fixed, it follows that
	\begin{equation}\label{Equa60}
	\left \langle I'_{\lambda}(v_{n}),\frac{v}{\|v\|}\right \rangle\leq \frac{C}{\sqrt{n}}(1+\|\xi'_{n}(0)\|). 
	\end{equation}
	
	To conclude that $ I'_{\lambda}(v_{n})\rightarrow 0 $ as $ n\rightarrow \infty $, it is enough to show that $ \{\xi'_{n}(0)\} $ is uniformly bounded in $ n $. In fact, consider $ G(u):= \langle I'_{\lambda}(u),u \rangle $ and note that
	\begin{equation*}
	\langle \xi'_{n}(0),v\rangle =\frac{\langle G'(v_{n}),v\rangle}{\langle G'(v_{n}),v_{n}\rangle}.
	\end{equation*}
	Note that $ G' $ maps bounded sets of $ W_{0}^{s,p}(\Omega) $ in bounded sets of $ W^{-s,p'}(\Omega) $. Moreover, as in the proof of Lemma \ref{lem5}, we have that
	\begin{equation*}
	\langle G'(v_{n}), v_{n}\rangle \leq (\gamma-q)\int _{\Omega}\frac{|v_{n}|^{q}}{|x|^{\alpha}}dx<0.
	\end{equation*}
	So, from Lemma \ref{Lem4}, we conclude that $ |\langle G'(v_{n}),v_{n} \rangle|\geq d $, for some $ d>0 $ and all $ n $. Therefore, $ \{\xi'_{n}(0)\} $ is uniformly bounded and from \eqref{Equa43}, \eqref{Equa60}, it follows that  $ \{v_{n}\} $ is a $ (PS)_{c_{V^{+}}} $ sequence.
\end{proof}

\subsection{The non-degenerate case}\label{Subnondeg}

In this subsection, we address the case non-degenerate $ M(0)=m_{0}>0 $. Assuming that $ M $ satisfies only \hyperlink{M1}{$ (M_{1}) $} and $ f $ satisfies \hyperlink{0}{$ (f_{0}) $}, \hyperlink{F1}{$ (F_{1}) $} and \hyperlink{F2}{$ (F_{2}) $}, we enunciate some technical results to truncated functional $ I_{\lambda,a} $. The proofs of these results it will be omitted because they are similar to proofs of the corresponding results presented in subsection \ref{Subdeg}. In fact, taking into account that $m_{0}\leq M_{a}(t)\leq a $, for all $ t\in \mathbb{R} $, it is enough replace \hyperlink{M2}{$ (M_{2}) $} by \hyperlink{M'2}{$ (M'_{2}) $}, \hyperlink{1}{$ (f_{1}) $} by \hyperlink{F1}{$ (F_{1}) $}, and conditions \hyperlink{2}{$ (f_{2}) $} and \hyperlink{3}{$ (f_{3}) $} by \hyperlink{F2}{$ (F_{2}) $}.

\begin{lemma}
	Functional $ I_{\lambda,a} $ satisfies the following geometric conditions:
	\begin{itemize}
		\item[(1)] for each $ u\in W_{0}^{s,p}(\Omega)\setminus \{0\} $, we have
		\begin{equation*}
		I_{\lambda,a}(tu^{+}+\theta u^{-})\rightarrow -\infty, \;\operatorname{as}\; |(t,\theta)|\rightarrow \infty;
		\end{equation*}
		\item[(2)] there exists $ R>0 $ such that 
		\begin{equation*}
			I_{\lambda,a}(u)\geq R\|u\|^{p}\; \operatorname{and}\; \langle I'_{\lambda,a}(u),u \rangle \geq R \|u\|^{p}.
		\end{equation*}
	\end{itemize}
\end{lemma}

\begin{lemma}
	For each $ u\in W_{0}^{s,p}(\Omega)\setminus \{0\} $, there exists unique $ t_{u}>0 $ such that $ t_{u}u\in \mathcal{N}_{\lambda,a} $. In particular, $ \mathcal{N}_{\lambda,a}\neq \emptyset $ and $ \mathcal{N}_{\lambda,a}^{\pm}\neq \emptyset $. Moreover, $ I_{\lambda,a}(t_{u}u)>I_{\lambda,a}(tu) $, for all $ t\geq 0 $, $ t\neq t_{u} $.
\end{lemma}
\begin{corollary}
	For each $ u\in W_{0}^{s,p}(\Omega)\setminus \{0\} $ with $ \langle I'_{\lambda,a}(u),u\rangle \leq 0 $, there exists unique $ t_{u}\in (0,1] $ such that $ t_{u}u\in \mathcal{N}_{\lambda,a} $.
\end{corollary}

\begin{lemma}
	For each $ w\in W_{0}^{s,p}(\Omega) $ with $ w^{+}\neq 0$ and $ w^{-}\neq 0 $, there exists unique pair $ (t_{w},\theta_{w})\in (0,\infty)\times (0,\infty) $ such that 
	\begin{equation*}
	t_{w}w^{+}+\theta _{w} w^{-}\in \mathcal{M}_{\lambda,a}.
	\end{equation*}
	In particular, $ \mathcal{M}_{\lambda,a}\neq \emptyset $. Furthermore, for all $ t,\theta \geq 0 $ with $ (t,\theta)\neq (t_{w},\theta _{w}) $, we have
	\begin{equation*}
	I_{\lambda,a}(tw^{+}+\theta w^{-})<I_{\lambda,a}(t_{w}w^{+}+\theta _{w}w^{-}).
	\end{equation*}
\end{lemma}

\begin{corollary}
	For each $ w\in W_{0}^{s,p}(\Omega) $ with $ w^{\pm}\neq 0 $ and $ \langle I'_{\lambda,a}(w), w^{\pm}\rangle \leq 0$, there exists unique pair $ (t_{w},\theta_{w})\in (0,1]\times (0,1] $ such that $ t_{w}w^{+}+\theta _{w}w^{-}\in \mathcal{M}_{\lambda,a} $.
\end{corollary}

\begin{lemma}\label{Lem21}
	\begin{itemize}
 	\item[(1)] There exists $ \kappa:=\kappa _{\lambda}>0 $ such that for all $ u\in \mathcal{N}_{\lambda,a} $,
	\begin{equation*}
		\|u\| \geq \kappa
	\end{equation*} 
and for all $ w\in \mathcal{M}_{\lambda,a} $,
\begin{equation*}
	\|w^{+}\|,\|w^{-}\|\geq \kappa;
\end{equation*}
	\item[(2)] For all $ u\in \mathcal{N}_{\lambda,a} $,
	\begin{equation*}
		I_{\lambda,a}(u)\geq \left (\frac{m_{0}}{p}-\frac{a}{\mu}\right )\|u\|^{p};
	\end{equation*}
	\item[(3)] If $ \{u_{n}\} \subset \mathcal{N}_{\lambda,a} $ and $ \{w_{n}\}\subset \mathcal{M}_{\lambda,a} $ are bounded sequences, then 
	\begin{equation*}
		\liminf_{n\rightarrow \infty} \int _{\Omega} \frac{|u_{n}|^{q}}{|x|^{\alpha}}dx>0\; \operatorname{and}\; \liminf_{n\rightarrow \infty} \int _{\Omega} \frac{|w_{n}|^{q}}{|x|^{\alpha}}dx>0.
	\end{equation*}
\end{itemize}
\end{lemma}

\begin{lemma}
	For each $ u\in \mathcal{N}_{\lambda,a} $, there exists $ \epsilon >0 $ and a differentiable function $ \xi_{a}: B(0,\epsilon)\rightarrow [0,\infty) $ such that $ \xi_{a} (0)=1 $, $ \xi_{a}(v)(u-v)\in \mathcal{N}_{\lambda,a} $ and 
	\begin{equation*}
		\langle \xi_{a}'(0),v\rangle =\frac{p \overline{M}_{a}'(\|u\|^{p})\langle (-\Delta)^{s}u,v \rangle-\lambda \int_{\Omega}\partial_{t}\overline{f}(x,u)vdx-q\int_{\Omega}\frac{|u|^{q-2}uv}{|x|^{\alpha}}dx}{p\overline{M}_{a}(\|u\|^{p})\|u\|^{p}-\lambda \int_{\Omega}\partial_{t}\overline{f}(x,u)u dx-q\int_{\Omega}\frac{|u|^{q}}{|x|^{\alpha}}dx}
	\end{equation*} 
\end{lemma}

\begin{lemma}
	For each $ w\in \mathcal{M}_{\lambda,a} $, we have that $ \det J_{(0,1)}\psi_{w,a}>0 $, where $ J_{(1,1)}\psi_{w,a} $ is the Jacobian matrix of $ \psi_{w,a} $ in pair $ (1,1) $.
\end{lemma}

\begin{proposition}\label{Prop21}
	The following asymptotic properties hold: 
	\begin{itemize}
		\item[(1)] For $ X_{\lambda,a}=\mathcal{N}_{\lambda,a},\mathcal{N}_{\lambda,a}^{+},\mathcal{N}_{\lambda,a}^{-},\mathcal{M}_{\lambda,a} $, it holds that $ c_{X_{\lambda,a}} $ is non-increasing in $ \lambda>0 $ and
		\begin{equation*}
		\lim_{\lambda\rightarrow \infty} c_{X_{\lambda,a}} =0.
		\end{equation*}
		\item[(2)] Let $ \{q_{n}\}\subset (p,p_{\alpha}^{*}] $ be such that $ q_{n}\rightarrow p_{\alpha}^{*} $ as $ n\rightarrow \infty $. Then,
		\begin{equation*}
		\lim_{\lambda\rightarrow \infty} \limsup_{n\rightarrow \infty}c_{\mathcal{M}_{\lambda,a,q_{n}}}=0. 
		\end{equation*}
	\end{itemize}
\end{proposition}

\begin{lemma}
	Let $ \{u_{n}\} $ and $ \{q_{n}\} $ be a sequences satisfying the hypotheses of Lemma \ref{lem4} such that $ I'_{\lambda,a,q_{n}}(u_{n})\rightarrow 0 $, as $ n\rightarrow \infty $. With the previous notations, if $ \nu_{j}>0 $ for some $ j\in \mathcal{J} $ or $  \nu _{k}^{\pm}>0 $ for some $ k\in \mathcal{J}^{\pm} $, we have
	\begin{equation*}
	\nu_{j}, \nu _{k}^{\pm}\geq (m_{0}S_{\alpha})^{\frac{N-\alpha}{sp-\alpha}}.
	\end{equation*}
\end{lemma}

\begin{proposition}[PS condition]\label{Prop22}
	The following statements are hold true:
	
	\noindent (1) for $ p<q<p_{\alpha}^{*} $, $ I_{\lambda,a} $ satisfies $ (PS)_{c} $, for all $ c\in \mathbb{R} $;
	
	\noindent (2) for $ q=p_{\alpha}^{*} $, $ I_{\lambda,a} $ satisfies $ (PS)_{c} $, for all
	\begin{equation*}
	c<\left (\frac{1}{\mu}-\frac{1}{p_{\alpha}^{*}}\right )(m_{0}S_{\alpha})^{\frac{N-\alpha}{sp-\alpha}}.
	\end{equation*}
\end{proposition}

\begin{lemma}
	There exists a closed set $ V_{a}^{\pm} $ with $ \mathcal{N}_{\lambda,a}^{\pm}\subset V_{a}^{\pm}\subset \mathcal{N}_{\lambda,a} $ and sequences $ \{v_{n}\}\subset V_{a}^{\pm} $, $ \{u_{n}\}\subset \mathcal{N}_{\lambda,a}^{\pm} $ such that $ \{v_{n}\} $ is a $ (PS)_{c_{V^{\pm}}} $ sequence and
	\begin{equation*}
	\lim_{n\rightarrow \infty}\|u_{n}-v_{n}\|=0.
	\end{equation*}
\end{lemma}

\section{Proof of main results}\label{Section3}

We consider the levels
\begin{equation*}
l_{1}:=\left (\frac{1}{\gamma}-\frac{1}{p_{\alpha}^{*}}\right )\min\left \{\left (M(1)S_{\alpha}\right )^{\frac{N-\alpha}{sp-\alpha}},\left (M(1)S_{\alpha}^{\frac{\gamma}{p}}\right )^{\frac{p_{\alpha}^{*}}{p_{\alpha}^{*}-\gamma}}\right \},\;\; l_{2}:=\left (\frac{1}{\mu}-\frac{1}{p_{\alpha}^{*}}\right )(m_{0}S_{\alpha})^{\frac{N-\alpha}{sp-\alpha}}
\end{equation*}
which are important in demonstrations of main results of this work.

The existence of ground state solution can be ensured by Mountain Pass Theorem (see \cite{AM}). In fact, let
\begin{equation*}
\Gamma_{\lambda}:=\{g\in C^{o}([0,1],W_{0}^{s,p}(\Omega)):g(0)=0, \,I_{\lambda}(g(1))<0\},
\end{equation*}
\begin{equation*}
\Gamma _{\lambda,a} :=\{g\in C^{o}([0,1],W_{0}^{s,p}(\Omega)):g(0)=0, \,I_{\lambda,a}(g(1))<0\}
\end{equation*}
and the minimax levels 
\begin{equation*}
c_{\lambda}:=\inf_{g\in \Gamma _{\lambda}}\sup_{t\in [0,1]}I_{\lambda}(g(t)),
\end{equation*}
\begin{equation*}
c_{\lambda,a}:=\inf_{g\in \Gamma _{\lambda,a}}\sup_{t\in [0,1]}I_{\lambda,a}(g(t)). 
\end{equation*}
Using Lemma \ref{Lem1.1} to $ I_{\lambda} $ and Lemma\ref{Lem21} to $ I_{\lambda,a} $, we have that both functionals $ I_{\lambda}$ and $ I_{\lambda,a} $ satisfy the geometric conditions of Mountain Pass Theorem:
\begin{itemize}
\item[$ \bullet $] there exist positive constants  $ d $ and $ d' $ such that $ J(u)\geq d $ for all $ u\in W_{0}^{s,p}(\Omega) $, with $ \|u\|=d' $;
\item[$ \bullet $] there exists $ e\in W_{0}^{s,p}(\Omega) $ with $ \|e\|>d' $ such that $ J(e)<0 $.
\end{itemize}

In the subcritical case, $ p<q<p_{\alpha}^{*} $, since both functionals $ I_{\lambda} $ and $ I_{\lambda,a} $ satisfy the $ (PS)_{c} $ condition, for all $ c\in \mathbb{R} $, we can apply Mountain Pass Theorem to conclude that $ c_{\lambda} $ and $ c_{\lambda,a} $ are positive critical values of $ I_{\lambda} $ and $ I_{\lambda,a} $, respectively.

In critical case, $ q=p_{\alpha}^{*} $, we apply the version of Mountain Pass Theorem without the (PS) condition (see \cite[Theorem 2.2]{Bre}) for $ I_{\lambda} $ and $ I_{\lambda,a} $ which together with item (3) of Proposition \ref{Prop3} and item (2) of Proposition \ref{Prop22}, imply that $ c_{\lambda} $ and $ c_{\lambda,a} $ are positive critical values of $ I_{\lambda} $ and $ I_{\lambda,a} $, respectively, whenever $ c_{\lambda}< l_{1}$ and $ c_{\lambda,a}<l_{2} $.

It is not difficult to see that $ c_{\lambda}=c_{\mathcal{N}_{\lambda}} $ and  $ c_{\lambda,a}=c_{\mathcal{N}_{\lambda,a}} $ for all $ \lambda>0 $. This ensures that any solution obtained in the minimax level is also a ground state solution.

\subsection{Proof of Theorem \ref{Theo11}}

\noindent \textbf{(1)} \emph{Signed solutions:} In order to obtain a positive solution and a negative one for \eqref{Equa1}, we will show that $ c_{\mathcal{N}_{\lambda}^{+}} $ and $ c_{\mathcal{N}_{\lambda}^{-}} $ are achieved by critical points of $ I_{\lambda} $. In fact, let $ \{v_{n}\}\subset V^{+} $ and $ \{u_{n}\}\subset \mathcal{N}_{\lambda}^{+} $ be obtained in Lemma \ref{Lem7}. By Proposition \ref{Prop3}, there exists $ u_{\lambda,1}\in W_{0}^{s,p}(\Omega) $ such that, up to subsequence, $ v_{n}\rightarrow u_{\lambda,1} $ as $ n\rightarrow \infty $ which implies $ I_{\lambda}(u_{\lambda,1})=c_{V^{+}} $ and $ I'_{\lambda}(u_{\lambda,1})=0 $. From \eqref{Equa64}, $ u_{n}\rightarrow u_{\lambda,1} $ as $ n\rightarrow \infty $. Since $ \mathcal{N}_{\lambda}^{+} $ is closed, it follows that $ u_{\lambda,1} \in \mathcal{N}_{\lambda}^{+} $ and, thus, $ I_{\lambda}(u_{\lambda,1})=c_{\mathcal{N}_{\lambda}^{+}} $. Analogously, using $ \mathcal{N}_{\lambda}^{-} $, it is shown that there exists a negative solution  $ u_{\lambda,2} $ for problem \eqref{Equa1} with $ I_{\lambda}(u_{\lambda,2})=c_{\mathcal{N}_{\lambda}^{-}} $.

It remains to show that at least one of solutions $ u_{\lambda,1} $ or $ u_{\lambda,2} $ is a ground state solution. In fact, let $ u_{0} $ be a arbitrary ground state solution of \eqref{Equa1}. Then, $ u_{0}\in\mathcal{N}_{\lambda} $ with $ I_{\lambda}(u_{0})=c_{\mathcal{N}_{\lambda}} $ and $ I'_{\lambda}(u_{0})=0 $. We can observe that $ u_{0} $ has constant sign. Otherwise, $ u_{0}\in \mathcal{M}_{\lambda} $ and by Remark \ref{Rem1}, it follows that
\begin{equation*}
c_{\mathcal{N}_{\lambda}}=I_{\lambda}(u_{0})\geq c_{\mathcal{M}_{\lambda}}\geq 2 c_{\mathcal{N}_{\lambda}},
\end{equation*}
which is a contradiction. Therefore, if $ u_{0} $ is a positive solution, we have
\begin{equation*}
c_{\mathcal{N}_{\lambda}^{+}}\leq I_{\lambda}(u_{0})=c_{\mathcal{N}_{\lambda}}\leq c_{\mathcal{N}_{\lambda}^{+}},
\end{equation*}
which implies that $ u_{\lambda,1} $ is a ground state solution. If $ u_{0} $ is a negative solution, we obtain that $ u_{\lambda,2} $ is a ground state solution.

\noindent \emph{Sign-changing solution:} Let $ \{w_{n}\}\subset \mathcal{M}_{\lambda} $ be a sequence such that 
\begin{equation*}
\lim _{n\rightarrow \infty} I_{\lambda}(w_{n})=c_{\mathcal{M}_{\lambda}}.
\end{equation*}
From Lemma \ref{Lem4} and the boundedness of $ \{I_{\lambda}(w_{n})\} $, it follows that $ \{w_{n}\} $ is a bounded sequence in $ W_{0}^{s,p}(\Omega) $. Then, up to a subsequence, still denoted by $ \{w_{n}\} $, there exists $ w\in W_{0}^{s,p}(\Omega) $ such that $ w_{n}\rightharpoonup w $ and since the embedding $ W_{0}^{s,p}(\Omega) \hookrightarrow L^{r}(\Omega,\frac{dx}{|x|^{\beta}})$ is compact, for all $ r\in [1,p_{\beta}^{*}) $ and all $ 0\leq\beta<ps $, we have that
\begin{equation*}
\begin{array}{lllc}
w_{n}\rightarrow w,\, \operatorname{in}\, L^{r}(\Omega,\frac{dx}{|x|^{\beta}}),\\
w_{n}(x)\rightarrow w(x), \, \operatorname{a.e.} \, x\in \Omega.
\end{array}
\end{equation*}
Furthermore, across simple arguments, we deduce
\begin{equation*}
\begin{array}{llllc}
w_{n}^{\pm}\rightharpoonup w^{\pm},\,\operatorname{in} \, W_{0}^{s,p}(\Omega),\\
w_{n}^{\pm}\rightarrow w^{\pm},\, \operatorname{in}\, L^{r}(\Omega,\frac{dx}{|x|^{\beta}}),\\
w_{n}^{\pm}(x)\rightarrow w^{\pm}(x), \, \operatorname{a.e.} \, x\in \Omega.
\end{array}
\end{equation*}

From Lemma \ref{Lem4}, since $ q<p_{\alpha}^{*} $, we have
\begin{equation*}
\int _{\Omega}\frac{|w^{\pm}|^{q}}{|x|^{\alpha}}dx=\liminf_{n\rightarrow \infty}\int _{\Omega}\frac{|w_{n}^{\pm}|^{q}}{|x|^{\alpha}}dx>0
\end{equation*}
which implies that $ w^{+}\neq 0 $ and $ w^{-}\neq 0 $. Also, from \hyperlink{M1}{$( M_{1}) $}, Fatou's Lemma and Lemma \ref{Lemm1}, it follows that 
\begin{equation*}
\langle I'_{\lambda}(w),w^{\pm}\rangle\leq \lim\limits_{n\rightarrow\infty}\langle I'_{\lambda}(w_{n}),w_{n}^{\pm}\rangle=0.
\end{equation*}
Then, Corollary \ref{Cor2} implies that there exists $ (t_{w},\theta_{w})\in (0,1]\times (0,1] $ such that $ t_{w}w^{+}+\theta_{w}w^{-}\in \mathcal{M}_{\lambda} $. Thus, by \eqref{R2} and \eqref{propf2}, we get
\begin{eqnarray}\label{Equa48}
c_{\mathcal{M}_{\lambda}}\leq I_{\lambda}(t_{w}w^{+}+\theta_{w}w^{-})&=&I_{\lambda}(t_{w}w^{+}+\theta_{w}w^{-})-\frac{1}{\gamma}\langle I'_{\lambda}(t_{w}w^{+}+\theta_{w}w^{-}),t_{w}w^{+}+\theta_{w}w^{-}\rangle\nonumber\\
&\leq &I_{\lambda}(w)-\frac{1}{\gamma}\langle I'_{\lambda}(w),w\rangle\\
&\leq&\lim_{n\rightarrow \infty}\left (I_{\lambda}(w_{n})-\frac{1}{\gamma}\langle I'_{\lambda}(w_{n}),w_{n}\rangle\right )\nonumber\\
&=&\lim_{n\rightarrow \infty} I_{\lambda}(w_{n})=c_{\mathcal{M}_{\lambda}}\nonumber
\end{eqnarray}
and when $ t_{w}<1 $ or $ \theta_{w}<1 $, inequality \eqref{Equa48} is strict. Therefore, $ t_{w}=\theta_{w}=1 $ which implies that $ w\in \mathcal{M}_{\lambda} $ with $ I_{\lambda}(w)=c_{\mathcal{M}_{\lambda}} $.

It remains to be shown that $ w $ is a critical point for $ I_{\lambda} $. Suppose, by contradiction, that $ I'_{\lambda}(w)\neq 0 $. Then, there exists $ v\in W_{0}^{s,p}(\Omega) $ (which by density can be taken in $ C_{c}^{\infty}(\Omega) $) such that $ \langle I'_{\lambda}(w),v\rangle<-1 $ and, by continuously, there exists $ \tau_{0}>0 $ sufficiently small, such that
\begin{equation}\label{Equa49}
\langle I'_{\lambda}(tw^{+}+\theta w^{-}+\tau v),v\rangle<-1, 
\end{equation}
whenever $ (t,\theta)\in \overline{B}((1,1),\tau_{0}) $ and $ |\tau| \leq\tau_{0}$. Now, by Lemma \ref{lem3}, 
\begin{equation}\label{Equa50}
\Psi_{w}(t,\theta)=\left (\left \langle I'_{\lambda}(tw^{+}+\theta w^{-}),tw^{+}\right \rangle,\left \langle I'_{\lambda}(tw^{+}+\theta w^{-}),\theta w^{-}\right \rangle\right )\neq (0,0),
\end{equation}
for all $ (t,\theta)\in \partial D $, where $ D=B((1,1),\tau_{0}) $. We consider $ h(\tau,t,\theta)=tw^{+}+\theta w^{-}+\tau v $ which satisfies $ h^{\pm}\neq 0 $ if $ (t,\theta)\in \overline{D} $ and $ 0\leq\tau\leq \tau_{1} $ with $ 0<\tau_{1}<\tau_{0} $ sufficiently small. We also consider the homotopy $ H:[0,\tau_{1}]\times \overline{D}\rightarrow \mathbb{R}^{2} $ defined by
\begin{equation*}
H_{\tau}(t,\theta):=\left (I'_{\lambda}(h(\tau,t,\theta)),h(\tau,t,\theta)^{+},I'_{\lambda}(h(\tau,t,\theta)),h(\tau,t,\theta)^{-}\right ).
\end{equation*}
Note that $ H_{0}=\Psi_{w} $ and by \eqref{Equa50}, $ \inf_{\partial D}|H_{0}|>0 $. Moreover, by uniform continuously, we obtain $ \inf_{\partial D} |H_{\tau}|>0$, for all $ 0\leq\tau \leq \tau_{1}$ with $ \tau_{1} $ sufficiently small. Thus, by Lemma \ref{Lem6}
\begin{equation*}
\deg(\Psi_{w},D,(0,0))= \operatorname{sgn}\det J_{(1,1)} \Psi_{w}=1
\end{equation*}
and from invariance under homotopy, we obtain
\begin{equation*}
\deg(H_{\tau_{1}},D,(0,0))=\deg(H_{0},D,(0,0))=1.
\end{equation*} 
Therefore, there exists $ (t_{1},\theta_{1})\in D $ such that $ H_{\tau_{1}}(t_{1},\theta_{1})=(0,0) $, namely, $ h(\tau_{1},t_{1},\theta_{1})\in \mathcal{M}_{\lambda} $. Applying \eqref{Equa49} and Lemma \ref{lem3}, we have
\begin{eqnarray*}
	c_{\mathcal{M}_{\lambda}}\leq I_{\lambda}(h(\tau_{1},t_{1},\theta_{1}))&=&I_{\lambda}(t_{1}w^{+}+\theta_{1}w^{-})+\int_{0}^{\tau_{1}}\langle I'_{\lambda}(t_{1}w^{+}+\theta_{1}w^{-}+tv),v\rangle dt\\
	&\leq&I_{\lambda}(t_{1}w^{+}+\theta_{1}w^{-})-\tau_{1}\\
	&\leq&I_{\lambda}(w)-\tau_{1}=c_{\mathcal{M}_{\lambda}}-\tau_{1}
\end{eqnarray*}
which is a contradiction. Hence, $ I'_{\lambda}(w)=0 $. Consequently, $ u_{\lambda,3}:=w $ is a least energy sign-changing solution for \eqref{Equa1} and by Remark \ref{Rem1}, we have that
\begin{equation*}
	I_{\lambda}(u_{\lambda,3})>I_{\lambda}(u_{\lambda,1})+I_{\lambda}(u_{\lambda,2}).
\end{equation*}  

\noindent \textbf{(2)} \emph{Signed solutions:} From Proposition \ref{Prop1} there exists $ \lambda^{*}>0 $ such that for all $ \lambda\geq \lambda^{*} $, we have
\begin{equation*}
	c_{\mathcal{N}_{\lambda}^{\pm}}<l_{1}.
\end{equation*}
Let $\{V_{\lambda}^{\pm}\} $ be obtained in Lemma \ref{Lem7}. Since $ c_{V_{\lambda}^{\pm}}\leq c_{\mathcal{N}_{\lambda}^{\pm}} $, by Proposition \ref{Prop3}, it follows that $ I_{\lambda} $ satisfies $ (PS)_{c_{V_{\lambda}^{\pm}}} $, for all $ \lambda \geq \lambda^{*} $. As in item (1), applying Lemma \ref{Lem7}, we get a positive solution $ u_{\lambda,1} $ and a negative one $ u_{\lambda,2} $ such that one of them is a ground state solution, for all $ \lambda\geq \lambda^{*} $.

\noindent\emph{Sign-changing solution:} Let $ \{q_{n}\} \subset (\mu,p_{\alpha}^{*}) $ such that $ q_{n}\rightarrow p_{\alpha}^{*} $ as $ n\rightarrow \infty $. By Proposition \ref{Prop1}, there exists $ \lambda^{*}>0 $ such that for all $ \lambda \geq \lambda^{*} $
\begin{equation}\label{Equa56}
c_{\mathcal{N}_{\lambda,p_{\alpha}^{*}}}<l_{1}
\end{equation} 
and
\begin{equation}\label{Equa57}
\limsup _{n\rightarrow \infty}c_{\mathcal{M}_{\lambda,q_{n}}}<c_{\mathcal{N}_{\lambda,p_{\alpha}^{*}}}+l_{1}.
\end{equation}
Fix $ \lambda\geq \lambda^{*} $. Then, from item (1) with $ q=q_{n} $, there exists a sequence $ \{w_{q_{n}}\}\subset \mathcal{M}_{\lambda,q_{n}} $ such that $ I_{\lambda,q_{n}}(w_{q_{n}})=c_{\mathcal{M}_{\lambda,q_{n}}} $ and $ I'_{\lambda,q_{n}}(w_{q_{n}})=0 $. By Lemma \ref{Lem4}, we have
\begin{equation*}
\left (\frac{1}{\gamma}-\frac{1}{\mu}\right )M(1)\min\{1,\|w_{q_{n}}\|^{\frac{\gamma-p}{p}}\}\|w_{q_{n}}\|^{p}\leq I_{\lambda,q_{n}}(w_{q_{n}})\leq C
\end{equation*}
which implies in the boundedness of $ \{w_{q_{n}}\} $. Thus, up to a subsequence, $ w_{q_{n}}\rightharpoonup w \in W_{0} ^{s,p}(\Omega)$ and
\begin{equation*}
\begin{array}{lllc}
w_{q_{n}}\rightarrow w,\, \operatorname{in}\, L^{r}(\Omega),\; r\in [1,p^{*})\\
w_{q_{n}}(x)\rightarrow w(x), \, \operatorname{a.e.} \, x\in \Omega,\\
\|w_{q_{n}}\|\rightarrow \tau_{0}\geq 0.
\end{array}
\end{equation*}
Since $ \{w_{q_{n}}^{\pm}\} $ is also bounded, an elementary argument shows that 
\begin{equation*}
\begin{array}{lllc}
w_{q_{n}}^{\pm}\rightharpoonup w^{\pm},\, \operatorname{in}\, W_{0}^{s,p}(\Omega),\\
w_{q_{n}}^{\pm}\rightarrow w^{\pm},\, \operatorname{in}\, L^{r}(\Omega),\\
w_{q_{n}}(x)\rightarrow w(x), \, \operatorname{a.e.} \, x\in \Omega.\\
\end{array}
\end{equation*}

Next we will prove that $ w_{n}\rightarrow w $ strongly in $ W_{0}^{s,p}(\Omega) $. Applying the concentration compactness principle to $ \{w_{q_{n}}\} $ and to both $ w_{q_{n}}^{\pm} $, we obtain for $ w $ measures $ \nu, \sigma $ and a collection at most countable $ \{x_{j}\}_{j\in \mathcal{J}}\subset \overline{\Omega} $ satisfying \eqref{Eq1}-\eqref{Eq4} and for $ w^{\pm} $, correspondents $ \nu^{\pm} $, $ \sigma^{\pm} $, $ \{x_{j}^{\pm}\}_{j\in \mathcal{J}^{\pm}} $.

Since
\begin{eqnarray*}
	c_{\mathcal{M}_{\lambda,q_{n}}}&=&I_{\lambda,q_{n}}(w_{q_{n}})-\frac{1}{\gamma}\langle I'_{\lambda,q_{n}}(w_{q_{n}}),w_{q_{n}}\rangle\\
	&=&\frac{1}{p}\widehat{M}(\|w_{q_{n}}\|^{p})-\frac{1}{\gamma}M(\|w_{q_{n}}\|^{p})\|w_{q_{n}}\|^{p}\\
	&&+\lambda \int _{\Omega}\frac{1}{\gamma}f(x,w_{q_{n}})w_{q_{n}}-F(x,w_{q_{n}})dx +\left (\frac{1}{\gamma}-\frac{1}{q_{n}}\right )\int _{\Omega}\frac{|w_{q_{n}}|^{q_{n}}}{|x|^{\alpha}}dx,
\end{eqnarray*}
from \eqref{R2} and \eqref{propf2}, we have
\begin{eqnarray}\label{Equa58}
\limsup_{n\rightarrow \infty}c_{\mathcal{M}_{\lambda,q_{n}}}&\geq& \frac{1}{p}\widehat{M}(\|w\|^{p})-\frac{1}{\gamma}M(\|w\|^{p})\|w\|^{p}\nonumber\\
&&+\lambda \int _{\Omega}\frac{1}{\gamma}f(x,w)w-F(x,w)dx+\left (\frac{1}{\gamma}-\frac{1}{p_{\alpha}^{*}}\right )(\nu ^{+}+\nu ^{-})(\mathbb{R}^{N}).
\end{eqnarray}

Note that $ w\neq 0 $. In fact, suppose that $ w=0 $, then $ \nu_{j}^{+}>0 $ for some $ j\in \mathcal{J}^{+} $ and $ \nu_{l}^{-}>0 $ for some $ l\in \mathcal{J}^{-} $ because otherwise, without loss of generality, say $ \nu_{j}^{+}=0 $ for all $ j\in \mathcal{J}^{+} $, by item (1) of Lemma \ref{Lem4}, we would have the contradiction
\begin{equation*}
M(\kappa^{p})\kappa^{p} \leq M(\kappa^{p})\lim_{n\rightarrow \infty}\|w_{q_{n}}^{+}\|^{p}\leq \lim_{n\rightarrow \infty} M(\|w_{q_{n}}\|^{p})\langle (-\Delta_{p})^{s}w_{q_{n}},w_{q_{n}}^{+}\rangle=\lim_{n\rightarrow \infty} \lambda \int_{\Omega}f(x,w_{q_{n}}^{+})w_{q_{n}}^{+}dx+\int _{\Omega}\frac{|w_{q_{n}}^{+}|^{q_{n}}}{|x|^{\alpha}}dx=0.
\end{equation*}
Thus, by \eqref{Equa58} and Lemma \ref{lem6}, we have
\begin{eqnarray*}
	\limsup_{n\rightarrow \infty} c_{\mathcal{M}_{\lambda,q_{n}}} &\geq& \left (\frac{1}{\gamma}-\frac{1}{p_{\alpha}^{*}}\right )\nu_{j}^{+}+\left (\frac{1}{\gamma}-\frac{1}{p_{\alpha}^{*}}\right )\nu_{l}^{-}\\
	&\geq&2l_{1}
\end{eqnarray*}
which, from \eqref{Equa56}, contradicts \eqref{Equa57}. Therefore, $ w\neq 0 $. Moreover, from Lemmas \ref{lem1} and \ref{lem2}, we have
\begin{equation*}
\langle I'_{\lambda,p_{\alpha}^{*}}(w),w\rangle\leq \liminf_{n\rightarrow \infty}\langle I'_{\lambda,q_{n}}(w_{q_{n}}),w\rangle =0
\end{equation*}
and, thus, there exists $ t_{w}\in (0,1] $ such that $ t_{w}w\in \mathcal{N}_{\lambda,p_{\alpha}^{*}} $. 

Now, we can prove that $ \nu_{j}^{+}=0 $, for all $ j\in \mathcal{J}^{+} $ and $ \nu_{j}^{-}=0 $ for all $ j\in \mathcal{J}^{-} $. Supposing not, by \eqref{Equa58}, \eqref{M'a2}, \eqref{propf2} and Lemma \ref{lem6}, it follows that

\begin{eqnarray*}
	\limsup_{n\rightarrow \infty}c_{\mathcal{M}_{\lambda,q_{n}}}&\geq& I_{\lambda,p_{\alpha}^{*}}(t_{w}w)-\frac{1}{\gamma}\langle I'_{\lambda,p_{\alpha}^{*}}(t_{w}w),t_{w}w\rangle+ l_{1}\\
	&\geq& I_{\lambda,p_{\alpha}^{*}}(t_{w}w)+l_{1}\\
	&\geq& c_{\mathcal{N}_{\lambda,p_{\alpha}^{*}}}+l_{1}
\end{eqnarray*}
contradicting \eqref{Equa57}. Therefore, $ \nu_{j}^{+}=0 $, for all $ j\in \mathcal{J}^{+} $ and $ \nu_{j}^{-}=0 $ for all $ j\in \mathcal{J}^{-} $ which implies
\begin{equation*}
\int _{\Omega}\frac{|w_{q_{n}}^{\pm}|^{q_{n}}}{|x|^{\alpha}}dx \rightarrow \int_{\Omega}\frac{|w^{\pm}|^{p_{\alpha}^{*}}}{|x|^{\alpha}}dx.
\end{equation*}
Since $ \langle I'_{\lambda,q_{n}}(w_{q_{n}}),w_{q_{n}}-w\rangle=0 $, follows that $ \langle (-\Delta_{p})^{s}w_{q_{n}},w_{q_{n}}-w\rangle=o_{n}(1) $. Then, by Lemma \ref{lem1} and standard arguments, we conclude that $ w_{q_{n}}\rightarrow w $ in $ W_{0}^{s,p}(\Omega) $. Thus, $ w_{q_{n}}^{\pm}\rightarrow w^{\pm} $ in $ W_{0}^{s,p}(\Omega) $ and from \eqref{Equa59} with $ q=q_{n} $, we obtain
\begin{equation*}
\|w^{\pm}\|=\lim_{n\rightarrow \infty}\|w_{q_{n}}^{\pm}\|>0.
\end{equation*}
Therefore, by Lemma \ref{lem2}, we conclude that $ I'_{\lambda,p_{\alpha}^{*}}(w)=\lim_{n\rightarrow \infty}I'_{\lambda,q_{n}}(w_{q_{n}})=0 $ and $ w\in \mathcal{M}_{\lambda,p_{\alpha}^{*}} $. 

It remains to prove that $ w $ has minimum energy. Given $ v\in \mathcal{M}_{\lambda,p_{\alpha}^{*}} $, by Lemma \ref{lem3},  there exist $ t_{n},\theta_{n}>0 $ such that 
\begin{equation*}
v_{n}=t_{n}v^{+}+\theta_{n} v^{-}\in \mathcal{M}_{\lambda,q_{n}}.
\end{equation*}
Then, proceeding as in the proof of item (2) of Proposition \ref{Prop1}, we conclude that $ (t_{n},\theta_{n})\rightarrow (1,1) $ and since
\begin{eqnarray*}
	I_{\lambda,q_{n}}(v_{n})&=&\frac{1}{p}\widehat{M}(\|v_{n}\|^{p})-\frac{1}{\mu}M(\|v_{n}\|^{p})\|v_{n}\|^{p}+\int _{\Omega}\frac{1}{\mu}f(x,v_{n})v_{n}-F(x,v_{n})dx\\
	&&+\left (\frac{1}{\mu}-\frac{1}{q_{n}}\right )t_{n}^{q_{n}}\int _{\Omega}\frac{|v^{+}|^{q_{n}}}{|x|^{\alpha}}dx+\left (\frac{1}{\mu}-\frac{1}{q_{n}}\right )\theta_{n}^{q_{n}}\int _{\Omega}\frac{|v^{-}|^{q_{n}}}{|x|^{\alpha}}dx,
\end{eqnarray*}
from Dominated Convergence Theorem, it follows that $ I_{\lambda,q_{n}}(v_{n})\rightarrow I_{\lambda,p_{\alpha}^{*}}(v) $. Note that, by construction, $ I_{\lambda,q_{n}}(w_{q_{n}})\leq I_{\lambda,q_{n}}(v_{n}) $ which implies
\begin{equation*}
I_{\lambda,p_{\alpha}^{*}}(w)=\lim_{n\rightarrow \infty}I_{\lambda,q_{n}}(w_{q_{n}})\leq \lim_{n\rightarrow \infty}I_{\lambda,q_{n}}(v_{n})=I_{\lambda,p_{\alpha}^{*}}(v)
\end{equation*}
and, therefore, $ I_{\lambda,p_{\alpha}^{*}}(w)=c_{\mathcal{M}_{\lambda,p_{\alpha}^{*}}} $. Taking $ u_{\lambda,3}:=w $, we obtain the third solution as intended, for all $ \lambda \geq \lambda^{*} $.\hfill $\square$

\subsection{Proof of Theorem \ref{Theo21}}
 
 Replacing the technical results obtained to functional $ I_{\lambda} $ in subsection \ref{Subdeg} by the corresponding ones obtained to functional $ I_{\lambda,a} $ in subsection \ref{Subnondeg}, we can follow the same passes used in the proof of Theorem \ref{Theo11} to prove this result. \hfill $\square$

\subsection{Proof of Theorem \ref{Theo1}}

\noindent \textbf{(1)} The proof of this item is analogous to proof of item (1) of Theorem \ref{Theo11}. In fact, since $ M $ is non-degenerate, it is not necessary to use the property \eqref{R3} and, therefore, condition \hyperlink{1}{$ (f_{1}) $} can be substituted by \hyperlink{F1}{$ (F_{1}) $}.

\noindent \textbf{(2)} From item (1) of Theorem \ref{Theo21}, we obtain three solutions for problem \eqref{Equa40}, $ u_{\lambda,a,1} $, $ u_{\lambda,a,2} $ and $ u_{\lambda,a,3} $ for which \hyperlink{C2}{$ ( \mathcal{C}_{\lambda,a} ) $} holds.
By Proposition \ref{Prop21}, there exists $ \lambda^{*}>0 $ such that for all $ \lambda\geq \lambda^{*} $,
\begin{equation*}
	c_{\mathcal{N}_{\lambda,a}^{\pm}}, c_{\mathcal{M}_{\lambda,a}}\leq \delta \left (\frac{m_{0}}{p}-\frac{a}{\mu}\right ).
\end{equation*}
Thus, from Lemma \ref{Lem21}, we obtain $ \|u_{\lambda,a,i}\|^{p}\leq \delta $, $ i=1,2,3 $. By \eqref{Ma2}, it follows that $ M_{a}(\|u_{\lambda,a,i}\|^{p})=M(\|u_{\lambda,a,i}\|^{p}) $, $ i=1,2,3 $. Therefore, for all $ \lambda\geq \lambda^{*} $, $ u_{\lambda,i}:=u_{\lambda,a,i} $, $ i=1,2,3 $, are solutions of problem \eqref{Equa1}, for which \hyperlink{C}{$ (\mathcal{C}_{\lambda}) $} holds. 

\noindent \textbf{(3)} From item (2) of Theorem \ref{Theo21}, there exist $ \overline{\lambda}>0 $ such that for all $ \lambda\geq \overline{\lambda} $, problem \eqref{Equa40} has three solutions $ u_{\lambda,a,1} $, $ u_{\lambda,a,2} $ and $ u_{\lambda,a,3} $ for which \hyperlink{C2}{$ ( \mathcal{C}_{\lambda,a} ) $} holds. Also, from Proposition \ref{Prop21}, there exists $ \lambda^{**}\geq \overline{\lambda} $ such that for all $ \lambda\geq \lambda^{**} $, 
\begin{equation*}
	c_{\mathcal{N}_{\lambda,a}^{\pm}}, c_{\mathcal{M}_{\lambda,a}}\leq \delta \left (\frac{m_{0}}{p}-\frac{a}{\mu}\right ).
\end{equation*}
As in item (2) of this theorem, we can conclude that $ u_{\lambda,i}:=u_{\lambda,a,i} $, $ i=1,2,3 $, are solutions of \eqref{Equa1} for which \hyperlink{C}{$ (\mathcal{C}_{\lambda}) $} holds, for all $ \lambda\geq \lambda^{**} $.\hfill $\square$

\bibliographystyle{amsplain}

\end{document}